\documentclass[onefignum,onetabnum,final]{siamart171218}

\usepackage{graphicx}
\usepackage{tikz}
\usepackage{xparse}
\usepackage{pgfplots}
\usepackage{mathtools}
\usepackage{amssymb}
\usepackage{cleveref}
\usepackage{csl}
\usepackage{mleftright}
\usepackage{physics}
\usepackage{subcaption}
\usepackage{siunitx}
\usepackage{bbm}

\graphicspath{{./Figures/}}

\newenvironment{rktableau}[2][1.35]
	{\def\arraystretch{#1}\array{#2}}
	{\endarray}
\newcommand{\methodstretch}{\def\arraystretch{1.4}}

\def\*#1{\mathbf{#1}}
\def\^#1{\widehat{#1}}
\def\-#1{\overline{#1}}
\def\~#1{\widetilde{#1}}

\DeclareMathOperator{\diag}{diag}

\newcommand{\eye}[1]{I_{#1 \times #1}}
\renewcommand{\Re}{\mathbbm{R}}
\newcommand{\one}{\mathbbm{1}}
\NewDocumentCommand{\scalarstab}{O{\rho} O{\alpha}}{\mathcal{S}^{\textsc{1d}}_{#1,#2}}
\NewDocumentCommand{\matstab}{O{\rho} O{\alpha}}{\mathcal{S}^{\textsc{2d}}_{#1,#2}}

\newcommand{\fast}{\mathfrak{f}}
\newcommand{\slow}{\mathfrak{s}}
\newcommand{\pre}{\mathfrak{p}}
\newcommand{\cor}{\mathfrak{c}}
\NewDocumentCommand{\component}{m m m}{^{ \left\{ #1 \right\} #2 \IfBooleanT{#3}{T} }}
\NewDocumentCommand{\comp}{m O{} s}{\component{#1}{#2}{#3}}
\NewDocumentCommand{\F}{O{} s}{\component{\fast}{#1}{#2}}
\NewDocumentCommand{\FL}{O{\lambda} O{} s}{\component{\fast, #1}{#2}{#3}}
\RenewDocumentCommand{\S}{O{} s}{\component{\slow}{#1}{#2}}
\NewDocumentCommand{\SL}{O{\lambda} O{} s}{\component{\slow, #1}{#2}{#3}}
\NewDocumentCommand{\FF}{O{} s}{\component{\fast, \fast}{#1}{#2}}
\NewDocumentCommand{\FS}{O{} s}{\component{\fast, \slow}{#1}{#2}}
\NewDocumentCommand{\FSL}{O{\lambda} O{} s}{\component{\fast, \slow, #1}{#2}{#3}}
\NewDocumentCommand{\SF}{O{} s}{\component{\slow, \fast}{#1}{#2}}
\NewDocumentCommand{\SFL}{O{\lambda} O{} s}{\component{\slow, \fast, #1}{#2}{#3}}
\RenewDocumentCommand{\SS}{O{} s}{\component{\slow, \slow}{#1}{#2}}

\DeclareMathAlphabet{\mathbase}{OT1}{pzc}{m}{n}
\NewDocumentCommand{\Abase}{O{} s}{\mathbase{A}\component{\slow,\slow}{#1}{#2}}
\NewDocumentCommand{\abase}{O{} s}{\mathbase{a}\component{\slow}{#1}{#2}}
\NewDocumentCommand{\bbase}{O{} s}{\mathbase{b}\component{\slow}{#1}{#2}}
\NewDocumentCommand{\bhatbase}{O{} s}{\widehat{\mathbase{b}}\component{\slow}{#1}{#2}}
\NewDocumentCommand{\cbase}{O{} s}{\mathbase{c}\component{\slow}{#1}{#2}}
\NewDocumentCommand{\deltac}{O{} s}{{\Delta \mathbase{c}}\component{\slow}{#1}{#2}}
\NewDocumentCommand{\deltaC}{O{} s}{{\Delta \mathbase{C}}\component{\slow}{#1}{#2}}
\NewDocumentCommand{\Dbase}{O{} s}{\mathbase{D}\component{\slow}{#1}{#2}}
\NewDocumentCommand{\Tbase}{O{} s}{\mathbase{T}\component{\slow}{#1}{#2}}

\newcommand{\SPCMethod}{step predictor-corrector MRI-GARK}
\newcommand{\IPCMethod}{internal stage predictor-corrector MRI-GARK}
\newcommand{\SPCabbv}{SPC-MRI-GARK}
\newcommand{\IPCabbv}{IPC-MRI-GARK}

\pgfplotsset{compat=1.14}
\pgfplotsset{every axis/.append style={
	width=\linewidth,
	legend style={font=\tiny},
	x label style={font=\small},
	y label style={font=\small},
	tick label style={font=\small}
}}

\newcommand{\orderplot}[1]{%
	\pgfplotstableread[col sep=comma]{./Data/#1}{\table}
	\pgfplotstablegetcolsof{\table}
	\pgfmathtruncatemacro\numberofcols{\pgfplotsretval-1}
	\pgfplotsinvokeforeach{1,...,\numberofcols}{
		\pgfplotstablegetcolumnnamebyindex{##1}\of{\table}\to{\colname}
		\addplot table [y index=##1] {\table};
		\addlegendentryexpanded{\colname}
	}
}

\newcommand{\perfplot}[2]{%
	\begin{tikzpicture}
		\begin{loglogaxis}[
				xlabel={Error},
				ylabel={CPU time (s)},
				legend entries={#2},
				legend columns=2,
				legend style={at={(1,1.05)},anchor=south east}
			]
			\addplot table[col sep=comma,x index=1,y index=0] {./Data/#1};
			\addplot table[col sep=comma,x index=3,y index=2] {./Data/#1};
			\addplot table[col sep=comma,x index=5,y index=4] {./Data/#1};
		\end{loglogaxis}
	\end{tikzpicture}
}

\newcommand{\scalarstabplot}[2]{%
	\begin{subfigure}{{\ifreport 0.3 \else 0.23 \fi}\textwidth}
		\includegraphics[width=\linewidth]{{#1}/scalar_stability}
		\caption{#2}
	\end{subfigure}
}

\newcommand{\matstabplot}[2]{%
	\begin{subfigure}{{\ifreport 0.9 \else 0.3 \fi}\textwidth}
		\ifreport
			\includegraphics[width=\linewidth]{{#1}/matrix_stability}
		\else
			\includegraphics[width=\linewidth,trim=14cm 0 14cm 1.45cm,clip]{{#1}/matrix_stability}
		\fi
		\caption{#2}
	\end{subfigure}
}

\newsiamremark{remark}{Remark}
\newsiamremark{comment}{Comment}
\newsiamremark{example}{Example}

\newif\ifreport
\reporttrue

\headers{Coupled MRI-GARK}{S. Roberts}

\title{Coupled Multirate Infinitesimal GARK Schemes for Stiff Systems with Multiple Time Scales\thanks{Submitted to the editors June 6, 2019
\funding{This work was funded by awards NSF CCF--1613905, NSF ACI--1709727, AFOSR DDDAS 15RT1037, and by the Computational Science Laboratory at Virginia Tech}}}
\author{
	Steven Roberts\thanks{
		Virginia Polytechnic Institute and State University, 
		Computational Science Laboratory, Department of Computer 
		Science, 2202 Kraft Drive, Blacksburg, VA 24060, USA
		(\email{steven94@vt.edu}, \email{sarshar@vt.edu}, \email{sandu@cs.vt.edu})
	}
	\and Arash Sarshar\footnotemark[2]
	\and Adrian Sandu\footnotemark[2]
}

\begin{document}
	
\ifreport
	\csltitle{Coupled Multirate Infinitesimal GARK Schemes for Stiff Systems with Multiple Time Scales}
	\cslauthor{Steven Roberts, Arash Sarshar, and Adrian Sandu}
	\cslemail{steven94@vt.edu, sarshar@vt.edu, sandu@cs.vt.edu}
	\cslreportnumber{7}
	\cslyear{18}
	\csltitlepage
\fi

\maketitle
\begin{abstract}
Traditional time discretization methods use a single timestep for the entire system of interest and can perform poorly when the dynamics of the system exhibits a wide range of time scales. Multirate infinitesimal step (MIS) methods (Knoth and Wolke, 1998) offer an elegant and flexible approach to efficiently integrate such systems. The slow components are discretized by a Runge--Kutta method, and the fast components are resolved by solving modified fast differential equations. Sandu (2018) developed the Multirate Infinitesimal General-structure Additive Runge--Kutta (MRI-GARK) family of methods that includes traditional MIS schemes as a subset. The MRI-GARK framework allowed the construction of the first fourth order MIS schemes. This framework also enabled the introduction of implicit methods, which are decoupled in the sense that any implicitness lies entirely within the fast or slow integrations. It was shown by Sandu that the stability of decoupled implicit MRI-GARK methods has limitations when both the fast and slow components are stiff and interact strongly. This work extends the MRI-GARK framework by introducing coupled implicit methods to solve stiff multiscale systems. The coupled approach has the potential to considerably improve the overall stability of the scheme, at the price of requiring implicit stage calculations over the entire system. Two coupling strategies are considered. The first computes coupled Runge--Kutta stages before solving a single differential equation to refine the fast solution. The second alternates between computing coupled Runge--Kutta stages and solving fast differential equations. We derive order conditions and perform the stability analysis for both strategies. The new coupled methods offer improved stability compared to the decoupled MRI-GARK schemes. The theoretical properties of the new methods are validated with numerical experiments.
\end{abstract}

\section{Introduction}

In this paper, we consider the additively partitioned ordinary differential equation (ODE)
\begin{equation} 
	\label{eqn:multirate_additive_ode}
	 y'= f(t,y) = f\F(t,y) + f\S(t,y), \qquad y(t_0)=y_0 \in \Re^{d},
\end{equation}
where component $f\F$ represents the fast dynamics of the system, and component $f\S$ the slow dynamics.  This structure models a feature appearing in many dynamical systems of practical interest: multiple characteristic time scales.

Multirate time integration methods are designed to efficiently solve \cref{eqn:multirate_additive_ode} by using different timesteps for the fast and slow components.  First explored by Rice \cite{Rice_1960} and Andrus \cite{Andrus_1979,Andrus_1993}, the multirating strategy has been expanded to numerous types of traditional time integration methods.  This includes Runge--Kutta methods \cite{Sandu_2007_MR_RK2,Guenther_2001_MR-PRK,Guenther_1994_partition-circuits,Kvaerno_2000_stability-MRK,Kvaerno_1999_MR-RK,Sandu_2019_MR-GARK_High-Order,roberts2019implicit}, 
linear multistep methods  \cite{Gear_1984_MR-LMM,Kato_1999,Sandu_2009_MR_LMM},
Rosenbrock-W methods  \cite{Guenther_1997_ROW},
extrapolation methods \cite{Sandu_2013_extrapolatedMR,Engstler_1997_MR-extrapolation}, Galerkin discretizations \cite{Logg_2003_MAG1}, and combined multiscale methodologies \cite{Engquist_2005}.

Multirate infinitesimal step (MIS) methods, first proposed by Knoth and Wolke \cite{Knoth_1998_MR-IMEX}, and later extended by others \cite{Knoth_2014_MR-Euler,Schlegel_2009_RFSMR,Schlegel_2010_MR-imex,Schlegel_2011_MR-implementation,Wensch_2009_MIS}, introduce a new multirating philosophy in which the fast method solves a modified ODE that advances the solution between slow stages. While the slow system is solved discretely, the fast system can be solved with arbitrary small steps, hence the naming ``infinitesimal step''.  In \cite{Sandu_2016_GARK-MR}, G\"{u}nther and Sandu cast MIS methods into the General-structure Additive Runge--Kutta (GARK) framework.  This framework was subsequently leveraged by Sandu in \cite{Sandu_2018_MRI-GARK} to create the multirate infinitesimal GARK (MRI-GARK) class of methods. One step of an MRI-GARK method advances the solution from $t_n$ to $t_n + H$ by
\begin{subequations}
	\label{eqn:MRI-GARK-decoupled}
	\begin{align}
		\label{eqn:MRI-GARK-decoupled-first_stage}
		& Y_1 = y_n \\
		\label{eqn:MRI-GARK-decoupled-internal_ode}
		& \left\{ \begin{aligned}
			v_i(0) &= Y_{i}, \\
			T_{i} &= t_n + \cbase_{i} \, H, \\
			v_i' &= \deltac_{i} \, f\F \mleft(T_{i} + \deltac_{i} \, \theta, v_i \mright) + \sum_{j=1}^{i+1} \gamma_{i,j} \mleft( \tfrac{\theta}{H} \mright) \, f\S\mleft(T_j,Y_j\mright), \\
			& \quad \text{for } \theta \in [0, H], \\
			Y_{i+1} &= v_i(H), \qquad i = 1, \dots, s\S,
		\end{aligned} \right. \\
		\label{eqn:MRI-GARK-decoupled-solution}
		& y_{n+1} = Y_{s\S + 1},
	\end{align}
\end{subequations}
where $\cbase$ are the slow method abscissae, and the modified fast ODEs  $v_i' = \dots$ advance the solution between the slow stages. In \cite{Sandu_2018_MRI-GARK}, Sandu presents MRI-GARK methods \cref{eqn:MRI-GARK-decoupled} of orders up to four that are explicit or implicit in the fast and slow systems but are not coupled across partitions. This work also provides new techniques in investigating the stability of partitioned methods that we have adopted in our paper. Recent developments in the field include the work of Sexton and Reynolds \cite{Sexton_2018_RMIS} where a new structure for fast integration weights is considered and shown to help reduce order conditions; the ``relaxed MIS" methods derived retain the same order as traditional MIS and it is possible to pair them for error control and adaptivity purposes.

This work extends the MRI-GARK family \cite{Sandu_2018_MRI-GARK} to include implicit methods with coupling between slow and fast systems. We construct two new families of schemes designed to offer improved stability compared to decoupled MRI-GARK methods. The first family is \SPCMethod{} methods that perform a coupled discrete prediction over the entire timestep, then use that information to perform an infinitesimal step correction with small timesteps. This is similar to the strategy used by Rice in \cite{Rice_1960} and Savcenco in \cite{Savcenco_2007_stability,Savcenko_2007_MR-hyperbolic}.  The second family is \IPCMethod{} methods that alternate between coupled discrete predictor stages and infinitesimal step correction ones. \Cref{fig:MRI-GARK_dependencies} shows the differences between the new strategies and previous approaches. Order condition theories and stability analyses are developed for both new families of methods, and schemes up to order four are designed. Numerical experiments are employed to verify the theoretical findings.

\begin{figure}
	\centering
	\begin{subfigure}{.47\linewidth}
		\includegraphics[width=\linewidth]{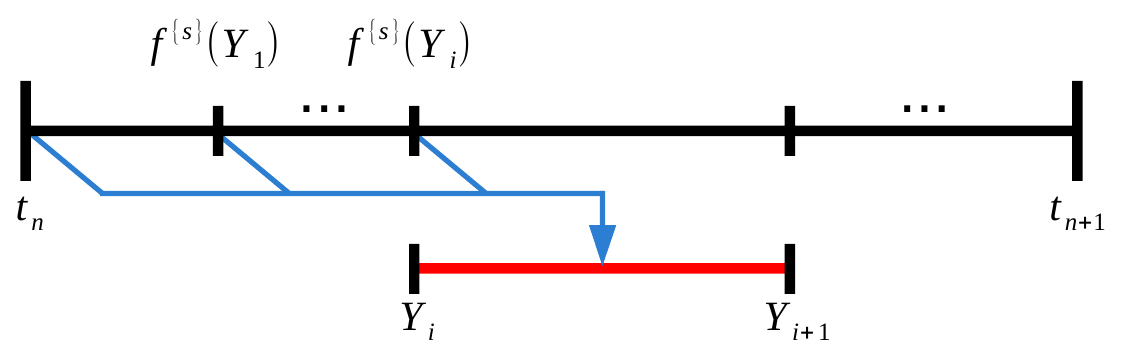}
		\caption{Traditional multirate infinitesimal schemes use previously computed stages, which means the slow tendencies are extrapolated in the formulation of modified fast systems.}
		\label{fig:MRI-GARK_dependencies:traditional}
	\end{subfigure}
	\hfill
	\begin{subfigure}{.47\linewidth}
		\includegraphics[width=\linewidth]{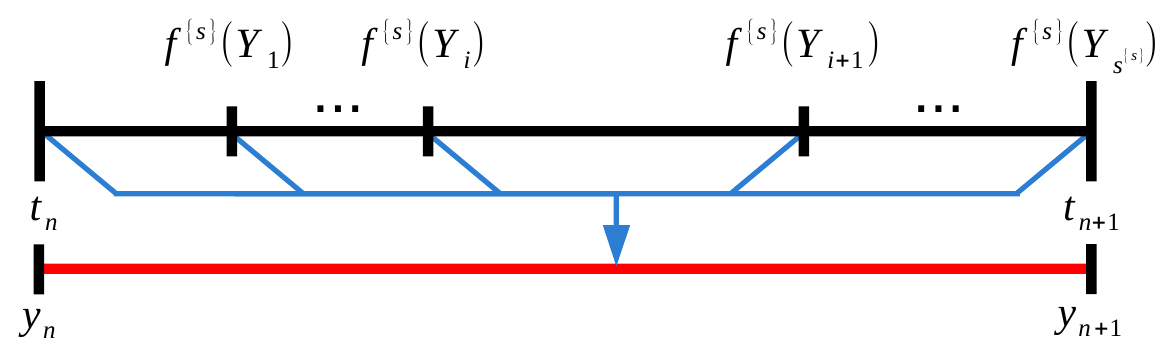}
		\caption{The newly proposed \SPCMethod{} schemes use all stages for computing slow tendencies and solve a single fast ODE over the entire step.}
		\label{fig:MRI-GARK_dependencies:SPC}
	\end{subfigure}
	\hfill
	\begin{subfigure}{.47\linewidth}
		\includegraphics[width=\linewidth]{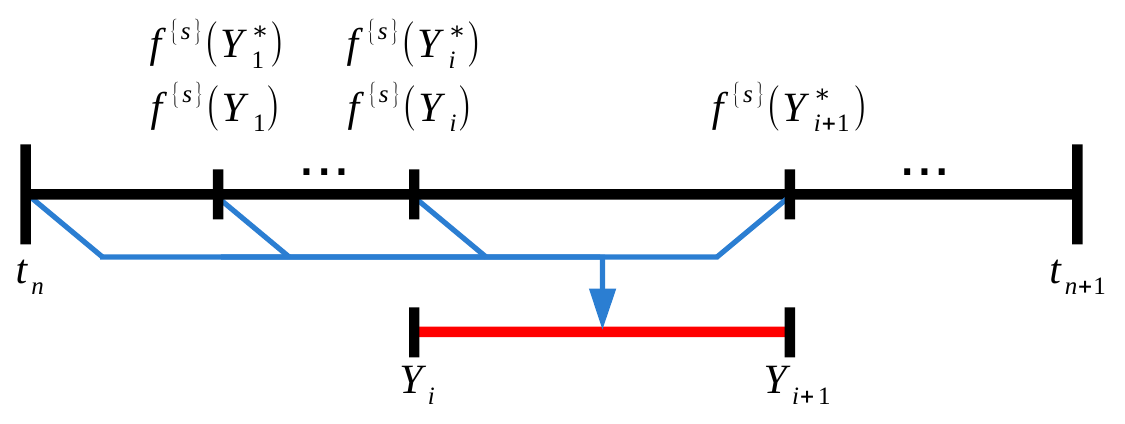}
		\caption{The newly proposed \IPCMethod{} schemes use previously computed stages and a predicted next stage, which allows for interpolation of slow tendencies in the formulation of modified fast systems.}
		\label{fig:MRI-GARK_dependencies:IPC}
	\end{subfigure}
	\caption{Comparison of MRI-GARK schemes: blue arrows indicate stage dependencies of the modified fast ODE and the red lines indicate the intervals over which a fast ODE is solved.}
	\label{fig:MRI-GARK_dependencies}
\end{figure}

The paper is organized as follows. The new family of step predictor corrector MRI-GARK schemes is introduced in \cref{sec:SC-MRI}, followed by its order condition theory and the stability analysis. In \cref{sec:IC-MR}, internal stage predictor corrector MRI-GARK schemes are defined and their order conditions and stability are established. Numerical results are reported in \cref{sec:numerics}, and concluding remarks are drawn in \cref{sec:conclusions}. \Cref{sec:new_methods} presents the lists of coefficients and stability plots for the newly developed methods.

\section{Step predictor-corrector MRI-GARK methods}
\label{sec:SC-MRI}

One coupling strategy commonly used in discrete multirate methods is a predictor-corrector approach, where the predictor evolves the entire system, while the corrector is only applied to the fast partition whose solution was ``predicted'' inaccurately (see \cite{Rice_1960,Savcenco_2007_stability,Savcenko_2007_MR-hyperbolic}).  First, a combined Runge--Kutta macro-step is taken which serves as the predictor.  The fast parts of the predicted stages are inaccurate and are refined by sub-stepping the fast component only. Approximations of the slow values needed during the micro-steps are obtained from interpolating the slow predicted values. The \SPCMethod{} methods, as depicted in \cref{fig:MRI-GARK_dependencies:SPC}, can be viewed as an extreme case of this coupling strategy where the multirate ratio is infinite, i.e., the corrector takes infinitely many steps to refine the fast solution.

\subsection{Method definition}

We start with a ``slow'' Runge--Kutta base method
\begin{equation}
	\label{eqn:slow_base_scheme}
	\begin{rktableau}{c|c}
		\cbase & \Abase \\
		\hline
		& \bbase* \\
		\hline
		& \bhatbase*
	\end{rktableau}
\end{equation}
with $s\S$ stages.  Unlike other multirate infinitesimal strategies, the base method is not restricted to be explicit or diagonally implicit.

\begin{definition}[Step predictor corrector MRI-GARK methods]
	One step of a \SPCMethod{} (\SPCabbv{}) scheme applied to \cref{eqn:multirate_additive_ode} is given by
	\begin{subequations}
		\label{eqn:SPC-MRI-GARK}
		\begin{align}
			\label{eqn:SPC-MRI-GARK_predictor}
			& Y_i = y_n + H \sum_{j=1}^{s\S} \, \abase_{i,j} f_j, \quad i =1,\dots,s\S, \\
			\label{eqn:SPC-MRI-GARK_corrector}
			& \left\{ \begin{aligned}
				v(0) &= y_n, \\
				v' &= f\F(t_n + \theta, v) + \sum_{j=1}^{s\S} \gamma_j \mleft( \tfrac{\theta}{H} \mright) \, f\S_j, \quad \text{for } \theta \in [0, H], \\
				y_{n+1} &= v(H),
			\end{aligned} \right.
		\end{align}
	\end{subequations}
	where $f_j \coloneqq f \big( t_n + \cbase_j \, H, Y_j \big)$ and $f_j\S \coloneqq f\S \big( t_n + \cbase_j \, H, Y_j \big)$.
\end{definition}

\begin{definition}[Slow tendency coefficients {\cite[Definition 2.2]{Sandu_2018_MRI-GARK}}]
	The time-dependent coefficients in \cref{eqn:SPC-MRI-GARK_corrector} are defined as polynomials:
	\begin{equation}
		\label{eqn:gamma_series}
		\gamma_i(t) \coloneqq \sum_{k \ge 0} \gamma_i^k \, t^k,
		\quad
		\~{\gamma}_i(t) \coloneqq \int_{0}^{t} \gamma_i(\tau) \, \dd \tau = \sum_{k \ge 0} \gamma_i^k \frac{t^{k+1}}{k+1},
		\quad
		\-{\gamma}_i \coloneqq \~{\gamma}_i(1).
	\end{equation}
\end{definition}

\begin{remark}[Embedded method]
	\label{rem:SPC-MRI-GARK_embedded}
	An embedded solution for an \SPCabbv{} method can be computed by solving the additional ODE
	\begin{equation*}
		\left\{ \begin{aligned}
			\^v(0) &= y_n, \\
			\^v' &= f\F \mleft( t_n + \theta, \^v \mright) + \sum_{j=1}^{s\S} \^\gamma_j \mleft( \tfrac{\theta}{H} \mright) \, f\S_j, \quad \text{for } \theta \in [0, H], \\
			\^y_{n+1} &= \^v(H),
		\end{aligned} \right.
	\end{equation*}
	which uses the embedded polynomials $\^\gamma_i$ and produces a solution of a different order.  We note that although this additional integration can be expensive, it can be done in parallel with \cref{eqn:SPC-MRI-GARK_corrector}.
\end{remark}

Consider the trivial partitioning $f\F = 0$, $f\S = f$ of \cref{eqn:multirate_additive_ode}.  In this case, it is natural to expect an \SPCabbv{} method to degenerate into the slow base method.  Note that the final solution of \cref{eqn:SPC-MRI-GARK_corrector} simplifies to
\begin{equation}
	\label{eqn:SPC-MRI-GARK_slow_only}
	y_{n+1}
	= y_n + \int_0^H \sum_{j=1}^{s\S} \gamma_j \mleft( \tfrac{\theta}{H} \mright) \, f\S_j \dd{\theta}
	= y_n + H \sum_{j=1}^{s\S} \-\gamma_j \, f_j\S.
\end{equation}
Thus, we enforce the condition
\begin{equation}
	\label{eqn:SPC-MRI-GARK_integral_condition}
	\bbase = \-\gamma.
\end{equation}

An important special case of \eqref{eqn:multirate_additive_ode} are component  partitioned systems:
\begin{equation}
	\label{eqn:ode_component}
	\begin{bmatrix} y\F \\ y\S \end{bmatrix}'
	= \begin{bmatrix} f\F \mleft( t, y\F, y\S \mright) \\ f\S \mleft( t, y\F, y\S \mright) \end{bmatrix}
	= \begin{bmatrix} f\F \mleft(t, y\F, y\S \mright) \\ 0 \end{bmatrix}
	+
	\begin{bmatrix} 0 \\ f\S \mleft( t, y\F, y\S \mright) \end{bmatrix}.
\end{equation}
One step of an \SPCabbv{} method \cref{eqn:SPC-MRI-GARK} applied to \cref{eqn:ode_component} reads:
\begin{subequations}
	\label{eqn:SPC-MRI-GARK_component}
	\methodstretch
	\begin{align}
		& \begin{bmatrix}
			Y\F_i \\ Y\S_i
		\end{bmatrix}
		= \begin{bmatrix}
			y\F_n + H \sum_{j=1}^{s\S} \abase_{i,j} \, f\F_j \\
			y\S_n + H \sum_{j=1}^{s\S} \abase_{i,j} \, f\S_j
		\end{bmatrix}, \\
		& \left\{ \begin{aligned}
			v\F(0) &= y\F_n, \\
			{v\F}' &= f\F \mleft(t_n + \theta, v, y\S_n + H \sum_{j=1}^{s\S} \~\gamma_j \mleft( \tfrac{\theta}{H} \mright) \, f_j\S \mright), \quad \text{for } \theta \in [0, H] \\
			\begin{bmatrix} y\F_{n+1} \\ y\S_{n+1} \end{bmatrix} &= \begin{bmatrix} v\F(H) \\ y\S_n + H \sum_{j=1}^{s\S} \bbase_j \, f\S_j \end{bmatrix},
		\end{aligned} \right. \label{eqn:SPC-MRI-GARK_component-ode}
	\end{align}
\end{subequations}
where $f\F_j \coloneqq f\F \big( t_n + \cbase_j \, H, Y\F_j, Y\S_j \big)$ and $f\S_j \coloneqq f\S \big( t_n + \cbase_j \, H, Y\F_j, Y\S_j \big)$.  With \cref{eqn:SPC-MRI-GARK_slow_only,eqn:SPC-MRI-GARK_integral_condition} the internal ODE integrates (and corrects) only the fast component, while the slow component is solved with the traditional base Runge--Kutta method \eqref{eqn:slow_base_scheme}.
\subsection{Order conditions}
\label{sec:SPC-MRI-GARK_order_conditions}

Following \cite{Sandu_2018_MRI-GARK}, we use an arbitrarily accurate Runge--Kutta method $\left( A\FF, b\F, c\F \right)$ to discretize the continuous ODE appearing in the method formulation, which casts the \SPCabbv{} scheme \eqref{eqn:SPC-MRI-GARK} into the GARK framework.  The discrete corrector stages, denoted $Y\FL[\cor]_i$, are computed as
\begin{align*}
	Y\FL[\cor]_i &= y_n + H \sum_{j=1}^{s\F} a\FF_{i,j} \left( f\FL[\cor]_j + \sum_{\ell=1}^{s\S} \gamma_\ell \big( c\F_j \big) \, f\S_\ell \right), \\
\ifreport
	&= y_n + H \sum_{j=1}^{s\F} a\FF_{i,j} \, f\FL[\cor]_j + H \sum_{\ell=1}^{s\S} \left( \sum_{j=1}^{s\F} a\FF_{i,j} \, \gamma_\ell \big( c\F_j \big) \right) f\S_\ell, \\
\fi
	&= y_n + H \sum_{j=1}^{s\F} a\FF_{i,j} \, f\FL[\cor]_j + H \sum_{j=1}^{s\S} \left( \sum_{k \geq 0} A\FF \, c\F[\times k] \, {\gamma^k}^T \right)_{i,j} f\S_j,
\end{align*}
where $f\FL[\cor]_j \coloneqq f\F \big( t_n + c\F_j \, H, Y\FL[\cor]_j \big)$ and the superscript $\times k$ denotes the elementwise vector power.  Similarly, the final solution reads
\begin{align*}
	y_{n+1}
	&= y_n + H \sum_{j=1}^{s\F} b\F_j \, f\FL[\cor]_j + H \sum_{j=1}^{s\S} \left( \sum_{k \geq 0} b\F* \, c\F[\times k] \right) f\S_j, \\
\ifreport
	&= y_n + H \sum_{j=1}^{s\F} b\F_j \, f\FL[\cor]_j + H \sum_{j=1}^{s\S} \left( \sum_{k \geq 0} \frac{1}{k+1} \gamma_j^k \right) f\S_j, \\
	&= y_n + H \sum_{j=1}^{s\F} b\F_j f\FL[\cor]_j + H \sum_{j=1}^{s\S} \-\gamma_j \, f\S_j, \\
\fi
	&= y_n + H \sum_{j=1}^{s\F} b\F_j \, f\FL[\cor]_j + H \sum_{j=1}^{s\S} \bbase_j f\S_j.
\end{align*}
Now, the corresponding GARK tableau for an \SPCabbv{} method is
\begin{equation*}
	\begin{rktableau}{c|cc|c|c}
		\cbase & \Abase & 0 & \Abase & \cbase \\
		c\F & 0 & A\FF & \sum_{k \geq 0} A\FF \, c\F[\times k] \, {\gamma^k}^T & c\FS \\ \hline
		\cbase & \Abase & 0 & \Abase & \cbase \\ \hline
		& 0 & b\F* & \bbase*
	\end{rktableau},
\end{equation*}
with $c\FS = \sum_{k \geq 0} A\FF \, c\F[\times k] \, {\gamma^k}^T \, \one\S$.

\subsubsection{Internal consistency}

\begin{theorem}[Internal consistency conditions]
	An \SPCabbv{} method \cref{eqn:SPC-MRI-GARK} satisfies the ``internal consistency'' conditions
	\begin{equation*}
		\*c\SF = \*c\SS \equiv \cbase
		\quad \text{and} \quad
		\*c\FF = \*c\FS
	\end{equation*}
	for any fast method iff the following conditions hold:
	\begin{equation}
		\label{eqn:SPC-MRI-GARK_ic}
		{\gamma^0}^T \, \one\S = 1
		\quad \text{and} \quad
		{\gamma^k}^T \, \one\S = 0 \quad \forall k \geq 1.
	\end{equation}
\end{theorem}

\begin{proof}
	All internal consistency equations are automatically satisfied except for the following one, which needs to be imposed explicitly:
	\begin{align*}
		c\F &= \sum_{k \geq 0} A\FF \, c\F[\times k] \, {\gamma^k}^T \, \one\S.
	\end{align*}
	It is easy to confirm \cref{eqn:SPC-MRI-GARK_ic} is sufficient to satisfy this condition, and thus, internal consistency.  Since the equality must hold for all $A\FF$, it must hold when all $A\FF \, c\F[\times k]$ are linearly independent.  Matching powers of the left- and right-hand sides proves the necessity of \cref{eqn:SPC-MRI-GARK_ic}.
\end{proof}

If an \SPCabbv{} method has a slow base method \cref{eqn:slow_base_scheme} of order two, then internal consistency is sufficient to guarantee the method is order two \cite{Sandu_2015_GARK}.

\subsubsection{Fourth order conditions}

In this section, we derive order conditions of the \SPCabbv{} schemes for up to order four. First, we define a set of useful coefficients.
\begin{definition}[Some useful coefficients {\cite[Definition 3.3]{Sandu_2018_MRI-GARK}}]
\ifreport	
	Consider the ``bushy'' Butcher tree \cite{Hairer_book_I}
	\begin{equation}
		\mathfrak{t}_k \coloneqq [\underbrace{\tau,\dots,\tau}_{k \text{ times}}] \in T,
	\end{equation}
	where $\tau \in T$ is the tree of order one and $[\cdot]$ is the operation of joining subtrees by a root.  
\fi
	An arbitrarily accurate Runge--Kutta method $\left( A\FF, b\F, c\F \right)$ satisfies the following equations:
	\begin{equation}
		\label{eqn:zeta_omega_xi}
		\ifreport
			\begin{alignedat}{3}
				\zeta_k &\coloneqq \frac{1}{\gamma \mleft( [\mathfrak{t}_k] \mright)}
				&& = b\F* \, A\FF \, c\F[\times k] &&= \frac{1}{(k+1)(k+2)}, \\
				\omega_k &\coloneqq \frac{1}{\gamma \mleft( [\tau,\mathfrak{t}_k] \mright)}
				&& = \left( b\F \times c\F \right)^T A\FF \, c\F[\times k] &&= \frac{1}{(k+1)(k+3)}, \\
				\xi_k &\coloneqq \frac{1}{\gamma \mleft( [[\mathfrak{t}_k]] \mright)}
				&&= b\F* \, A\FF \, A\FF \, c\F[\times k] &&= \frac{1}{(k+1)(k+2)(k+3)}.
			\end{alignedat}
		\fi
	\end{equation}
\end{definition}

\begin{theorem}[Fourth order coupling conditions]
	\label{thm:order_conditions}
	An internally consistent \SPCabbv{} method \cref{eqn:SPC-MRI-GARK} satisfying \cref{eqn:SPC-MRI-GARK_integral_condition} has order four iff the slow base scheme \cref{eqn:slow_base_scheme} has order at least four, and the following coupling conditions hold:
	\begin{subequations}
		\label{eqn:SPC-MRI-GARK_oc}
		\begin{align}
			\label{eqn:SPC-MRI-GARK_oc:3a}
			\frac{1}{6} &= \sum_{k \geq 0} \zeta_k \, {\gamma^k}^T \, \cbase, & \text{(order 3)} \\
			\label{eqn:SPC-MRI-GARK_oc:4a}
			\frac{1}{8} &= \sum_{k \geq 0} \omega_k \, {\gamma^k}^T \, \cbase, & \text{(order 4)} \\
			\label{eqn:SPC-MRI-GARK_oc:4b}
			\frac{1}{12} &= \sum_{k \geq 0} \zeta_k \, {\gamma^k}^T \, \cbase[\times 2], & \text{(order 4)} \\
			\label{eqn:SPC-MRI-GARK_oc:4d}
			\frac{1}{24} &= \sum_{k \geq 0} \zeta_k \, {\gamma^k}^T \, \Abase \, \cbase. & \text{(order 4)}
		\end{align}
	\end{subequations}
	%
\end{theorem}

\begin{proof}
	An internally consistent GARK scheme is order four iff the base methods are order four and the 12 coupling conditions up to order four are satisfied \cite{Sandu_2015_GARK}. We proceed with checking each coupling condition.
	\paragraph{Condition 3a} The first third order condition gives \cref{eqn:SPC-MRI-GARK_oc:3a}:
	\begin{equation*}
		\frac{1}{6}
		= \*b\F* \, \*A\FS \, \*c\S
		= \sum_{k \geq 0} \zeta_k \, {\gamma^k}^T \, \cbase.
	\end{equation*}
	\paragraph{Condition 3b} The other third order condition is automatically satisfied if the slow base method is order three:
	\begin{equation*}
		\frac{1}{6}
		= \*b\S* \, \*A\SF \, \*c\F
		= \bbase* \, \Abase \, \cbase.
	\end{equation*}
	\paragraph{Condition 4a} The first fourth order condition gives \cref{eqn:SPC-MRI-GARK_oc:4a}:
	\begin{equation*}
		\frac{1}{8}
		= \left( \*b\F \times \*c\F \right)^T \*A\FS \, \*c\S
		= \sum_{k \geq 0} \omega_k \, {\gamma^k}^T \, \cbase.
	\end{equation*}
	\paragraph{Condition 4b} This condition is automatically satisfied if the slow base method is order four:
	\begin{equation*}
	\frac{1}{8}
		= \left( \*b\S \times \*c\S \right)^T \*A\SF \, \*c\F
		= \left( \bbase \times \cbase \right)^T \Abase \, \cbase.
	\end{equation*}
	\paragraph{Condition 4c} This condition proves \cref{eqn:SPC-MRI-GARK_oc:4b}:
	\begin{equation*}
		\frac{1}{12}
		= \*b\F* \, \*A\FS \, \*c\S[\times 2]
		= \sum_{k \geq 0} \zeta_k \, {\gamma^k}^T \, \cbase[\times 2].
	\end{equation*}
	\paragraph{Condition 4d} This condition is automatically satisfied if the slow base method is order four:
	\begin{equation*}
		\frac{1}{12}
		= \*b\S* \, \*A\SF \, \*c\F[\times 2]
		= \bbase* \, \Abase \, \cbase[\times 2].
	\end{equation*}
	\paragraph{Condition 4e} This condition is the redundant since it is the difference of Condition 3a and 4a:
	\begin{equation*}
		\frac{1}{24} = \*b\F* \, \*A\FF \, \*A\FS \, \*c\S
		= \sum_{k \geq 0} \xi_k \, {\gamma^k}^T \, \cbase.
	\end{equation*}
	\paragraph{Condition 4f} This condition proves \cref{eqn:SPC-MRI-GARK_oc:4d}:
	\begin{equation*}
		\frac{1}{24}
		= \*b\F* \, \*A\FS \, \*A\SF \, \*c\F
		= \sum_{k \geq 0} \zeta_k \, {\gamma^k}^T \, \Abase \, \cbase.
	\end{equation*}
	\paragraph{Condition 4g} The following condition is identical to Condition 4f:
	\begin{equation*}
		\frac{1}{24}
		= \*b\F* \, \*A\FS \, \*A\SS \, \*c\S
		= \sum_{k \geq 0} \zeta_k \, {\gamma^k}^T \, \Abase \, \cbase.
	\end{equation*}
	\paragraph{Condition 4h} This condition is automatically satisfied if the slow base method is order four:
	\begin{equation*}
		\frac{1}{24}
		= \*b\S* \, \*A\SS \, \*A\SF \, \*c\F
		= \bbase* \, \Abase \, \Abase \, \cbase.
	\end{equation*}
	\paragraph{Condition 4i} This condition is automatically satisfied if the slow base method is order four:
	\begin{equation*}
		\frac{1}{24}
		= \*b\S* \, \*A\SF \, \*A\FS \, \*c\S
		= \bbase* \, \Abase \, \Abase \, \cbase.
	\end{equation*}
	\paragraph{Condition 4j} This condition is automatically satisfied if the slow base method is order four:
	\begin{equation*}
		\frac{1}{24}
		= \*b\S* \, \*A\SF \, \*A\FF \, \*c\F
		= \bbase* \, \Abase \, \Abase \, \cbase.
	\end{equation*}
\end{proof}

\begin{remark}
	In the proof of \cref{thm:order_conditions}, all coupling order conditions that start with $\*b\S*$ collapse onto the order conditions of the slow base method.  In the context of two-trees \cite{SanzSerna_1997_symplectic,Sandu_2015_GARK}, trees containing both slow and fast nodes with a slow root can be recolored into purely slow trees.  The purely fast trees are of no concern since the fast base method is arbitrarily accurate.  The remaining trees contain slow and fast nodes with a fast root, which correspond to coupling conditions \cref{eqn:SPC-MRI-GARK_oc} that must be explicitly enforced through the choice of $\gamma$.
\end{remark}
\subsection{Stability analysis}

	%

\subsubsection{Scalar stability analysis}
%
Consider the partitioned, linear, scalar test problem
\begin{equation}
	\label{eqn:scalar_test_problem}
	y' =  \lambda\F \, y + \lambda\S \, y, \qquad \lambda\F, \lambda\S \in \mathbb{C}^-,
\end{equation}
and let $z\F \coloneqq H \, \lambda\F$, $z\S \coloneqq H \, \lambda\S$, and $z \coloneqq z\F + z\S$.  
Applying the \SPCabbv{} method \eqref{eqn:SPC-MRI-GARK} to  \cref{eqn:scalar_test_problem} 
\ifreport
yields
\begin{align*}
	& Y = \left(\eye{s\S} - z \, \Abase \right)^{-1} \one\S y_n,\\
	& \left\{ \begin{aligned}
		v(0) &= y_n, \\
		v' &= \lambda\F \, v + \lambda\S \sum_{j=1}^{s\S} \gamma_j \mleft( \tfrac{\theta}{H} \mright) \, Y_j, \\
		&= \lambda\F \, v + \lambda\S \sum_{k \geq 0} \left( \tfrac{\theta}{H} \right)^k {\gamma^k}^T \, Y, \\
		y_{n+1} &= \varphi_0 \big( z\F \big) \, y_n + z\S \, \mu \big( z\F \big)^T \, Y,
	\end{aligned} \right.
\end{align*}
which 
\fi
leads to the stability function
\begin{equation}
	\label{eqn:SPC-MRI-GARK_scalar_stability}
	\begin{split}
		y_{n+1} &= R \big( z\F, z\S \big) \, y_n, \\
		R\big( z\F, z\S \big) &\coloneqq \varphi_0 \big( z\F \big) + z\S \, \mu \big( z\F \big)^T \left( \eye{s\S} - z \, \Abase \right)^{-1} \one\S,
	\end{split}
\end{equation}
where, following \cite{Sandu_2018_MRI-GARK}:
\begin{equation*}
\begin{split}
	\mu\big( z\F \big) &\coloneqq \sum_{k \geq 0} \gamma^k \, \varphi_{k+1}\big( z\F \big), \\
	\varphi_0(z) &\coloneqq e^z, \qquad
	\varphi_{k+1}(z) \coloneqq \int_0^1 e^{z (1-t)} t^k \dd{t}
	= \begin{cases}
		\frac{e^z - 1}{z} & k = 0 \\
		\frac{k \varphi_{k}(z) - 1}{z} & k > 0
	\end{cases}.
\end{split}	
\end{equation*}
Of special interest are cases when a partition becomes infinitely stiff.  If the base method has bounded internal stability, the stability function \cref{eqn:SPC-MRI-GARK_scalar_stability} enjoys the following property:
\begin{subequations}
	\begin{equation}
		\label{eqn:SPC-MRI-GARK_zf_lim}
		\lim_{z\F \to - \infty} R\big( z\F, z\S \big) = 0.
	\end{equation}
	Provided $\Abase$ is invertible, e.g. the base method is SDIRK,
	\begin{equation}
		\label{eqn:SPC-MRI-GARK_zs_lim}
		\lim_{z\S \to - \infty} R\big( z\F, z\S \big) = \varphi_0 \big( z\F \big) - \mu \big( z\F \big)^T \left(\Abase\right)^{-1} \one\S.
	\end{equation}
\end{subequations}
Although \cref{eqn:SPC-MRI-GARK_zs_lim} cannot be zero for all $z\F$ due to the linear independence of $\varphi$ functions, its modulus is bounded for $z\F \in \mathbb{C}^-$.
%

\subsubsection{Matrix stability analysis}

Following \cite{Kvaerno_2000_stability-MRK,Sandu_2018_MRI-GARK}, consider the matrix test problem
\begin{equation}
	\label{eqn:matrix_test_problem}
	\begin{bmatrix} y\F \\ y\S \end{bmatrix}'
	= \begin{bmatrix}
		\lambda\F & \eta\S  \\
		\eta\F & \lambda\S
	\end{bmatrix}
	\begin{bmatrix} y\F \\ y\S \end{bmatrix}
	= \underbrace{
		\begin{bmatrix}
			\lambda\F & \frac{1-\xi}{\alpha} \left(\lambda\F-\lambda\S\right) \\
			-\alpha \, \xi \left(\lambda\F-\lambda\S\right) & \lambda\S
		\end{bmatrix}
	}_{\boldsymbol{\Omega}}
	\begin{bmatrix} y\F \\ y\S \end{bmatrix}.
\end{equation}
The change of variables that produces $\boldsymbol{\Omega}$ \cite{Sandu_2018_MRI-GARK}
\ifreport
\begin{equation*}
	\alpha \coloneqq \frac{\lambda\F-\lambda\S+\delta}{2 \, \eta\S}, \quad
	\xi  \coloneqq \frac{\lambda\F - \lambda\S - \delta }{2 \left(\lambda\F-\lambda\S\right)}, \quad
	\delta = \sqrt{4 \, \eta\F \, \eta\S+\left(\lambda\F-\lambda\S\right)^2},
\end{equation*}
\fi
allows the matrix eigenvalues to be written as linear combinations of the diagonal entries: $\xi \, \lambda\F + (1 -\xi) \, \lambda\S$ and $(1-\xi) \, \lambda\F + \xi \, \lambda\S$.  The coupling between the fast and slow variables is controlled by $\xi$.  Values close to zero indicate the slow system is weakly influenced by the fast one, while values close to one indicate the fast system is weakly influenced by the slow one.

Let
\begin{equation*}
	Z \coloneqq \begin{bmatrix}
		z\F & w\S \\ w\F & z\S
	\end{bmatrix} \coloneqq H \begin{bmatrix}
		\lambda\F & \eta\S  \\
		\eta\F & \lambda\S
	\end{bmatrix},
	\qquad
	\~\mu \big( z\F \big) \coloneqq \sum_{k \geq 0} \frac{\gamma^k}{k+1} \varphi_{k+2} \big( z\F \big).
\end{equation*}
The component partitioned \SPCabbv{} method \cref{eqn:SPC-MRI-GARK_component} applied to the matrix test problem \cref{eqn:matrix_test_problem} gives
\ifreport
\begin{align*}
	& \begin{bmatrix}
		Y\F \\ Y\S
	\end{bmatrix}
	= \left( \eye{2s\S} - Z \otimes \Abase \right)^{-1} \begin{bmatrix}
		y\F_n \, \one\S \\
		y\S_n \, \one\S
	\end{bmatrix}, \\
	& \left\{ \begin{aligned}
		v\F(0) &= y\F_n, \\
		{v\F}' &= \lambda\F \, v\F + \eta\S \, y\S_n + \eta\S \sum_{j=1}^{s\S} \~\gamma_j \mleft( \tfrac{\theta}{H} \mright) \left( w\F \, Y\F_j + z\S \, Y\S_j \right), \\
		&= \lambda\F \, v\F + \eta\S \, y\S + \eta\S \sum_{k \geq 0} \frac{(\theta/H)^{k+1}}{k+1} {\gamma^k}^T \left( w\F \, Y\F + z\S \, Y\S \right), \\
		\begin{bmatrix} y_{n+1}\F \\ y_{n+1}\S \end{bmatrix}
		&= \begin{bmatrix}
			v\F(H) \\
			y\S_n + w\F \, \bbase* \, Y\F + z\S \, \bbase* \, Y\S
		\end{bmatrix},
	\end{aligned} \right.
\end{align*}
We write the solution of the ODE as
\begin{equation*}
	y\F_{n+1} = \varphi_0 \big( z\F \big) \, y\F_n + w\S \, \varphi_1 \big( z\F \big) \, y\S_n + w\S \, \~\mu \big( z\F \big)^T \left( w\F \, Y\F + z\S \, Y\S \right).
\end{equation*}
The transfer matrix for the matrix test problem can be written as
\fi
\begin{equation}
	\begin{split}
		\begin{bmatrix} y\F_{n+1} \\ y\S_{n+1} \end{bmatrix}
		&= \*M \big( z\F, z\S, w\S, w\F \big) \begin{bmatrix} y\F_n \\ y\S_n \end{bmatrix}, \\
		\*M \big( z\F, z\S, w\S, w\F \big) &\coloneqq \begin{bmatrix}
			\varphi_0(z\F) & w\S \varphi_1 \mleft( z\F \mright) \\
			0 & 1
		\end{bmatrix} \\
		&\quad + \begin{bmatrix}
			w\S \, w\F \, \~\mu \mleft( z\F \mright)^T & w\S \, z\S \, \~\mu \mleft( z\F \mright)^T \\
			w\F \, \bbase* & z\S \, \bbase*
		\end{bmatrix} \mathfrak{Y} (Z),
	\end{split}
\end{equation}
where $\mathfrak{Y} (Z)$ is the internal stability matrix:
\begin{equation*}
	\mathfrak{Y} (Z) \coloneqq \left( \eye{2s\S} - Z \otimes \Abase \right)^{-1} \left( \eye{2} \otimes \one\S \right).
\end{equation*}
%

\subsection{Construction of practical \SPCabbv{} methods}

We develop new implicit \SPCabbv{} methods of up to order four. Their coefficients are presented in \cref{sec:SPC-MRI-GARK_methods}.  The base methods are chosen to be existing, high-quality schemes that have either singly diagonally implicit (SDIRK) or explicit first stage single diagonally implicit (ESDIRK) structures.  These offer a nice balance between stability and computational complexity. We note that explicit and fully implicit base methods can be employed as well.  The $\gamma(t)$ coupling coefficients for each method are determined by first enforcing the order conditions, and then using remaining free parameters to optimize for stability.  Plots of the scalar and matrix stability regions are provided in \cref{fig:SPC-MRI-GARK_scalar_stability,fig:SPC-MRI-GARK_matrix_stability}, respectively.  These regions are significantly larger than those of the decoupled MRI-GARK counterparts developed in \cite{Sandu_2018_MRI-GARK}.
\section{Internal stage predictor-corrector MRI-GARK methods}
\label{sec:IC-MR}

Traditional multirate infinitesimal methods subdivide the integration interval $[t_n, t_{n+1}]$ into subintervals $[t_n + \cbase_i \, H,t_n + \cbase_{i+1} \, H]$, and solve a fast ODE over each subinterval.  This advances the solution from one abscissa to the next, and then to the final solution.  As illustrated in \cref{fig:MRI-GARK_dependencies:IPC}, an \IPCMethod{} method follows this strategy, but also incorporates a predictor-corrector strategy similar to that used in \SPCabbv{} schemes.  On each subinterval, the solution is first predicted with a traditional Runge-Kutta stage calculation.  Next, the fast components are refined by solving an ODE which uses previous predictor and corrector stages, as well as the current predictor stage, to implement the slow tendencies.

\subsection{Method definition}

Again, we start with a slow Runge--Kutta base method \cref{eqn:slow_base_scheme}, but now enforce that is has a diagonally implicit structure and the abscissae are non-decreasing:
\begin{equation*}
	0 \leq \cbase_1 \leq \cbase_2 \leq \ldots \leq \cbase_{s\S} \leq 1.
\end{equation*}
This ensures that each ODE between stages is not integrated backward in time.  We define the abscissa increments:
\begin{equation*}
	\deltac = \left[
		\cbase_1, \cbase_2 - \cbase_1, \dots, \cbase_{s\S} - \cbase_{s\S - 1}
	\right]^T. 
\end{equation*}
The final integration from $\cbase_{s\S}$ to $1$ can introduce special cases that increase the complexity of the notation, order conditions, and stability analysis.  We will impose that the base slow method is stiffly accurate \cite{Hairer_book_II}, which makes the last stage equal to the final solution, and simplifies the subsequent analyses.  This comes at no loss of generality since we can always rewrite a Runge--Kutta method into a reducible, but stiffly accurate form.
\ifreport
In Butcher tableau notation we have:
\begin{equation*}
	\begin{rktableau}{c|c}
		\cbase & \Abase \\ \hline
		& \bbase*
	\end{rktableau}
	\quad \rightarrow \quad
	\begin{rktableau}{c|cc}
		\cbase & \Abase & 0\S \\
		1 & \bbase* & 0 \\ \hline
		& \bbase* & 0
	\end{rktableau}.
\end{equation*}
\fi
\begin{definition}[Internal stage predictor corrector MRI-GARK methods]
	One step of an \IPCMethod{} (\IPCabbv{}) scheme applied to \cref{eqn:multirate_additive_ode} is given by
	\begin{subequations}
		\label{eqn:IPC-MRI-GARK}
		\begin{align}
			\label{eqn:IPC-MRI-GARK_first_stage}
			& Y_0 \coloneqq y_n, \quad \cbase_0 \coloneqq 0, \\
			\label{eqn:IPC-MRI-GARK_internal_ode}
			& \left\{ \begin{aligned}
				Y^{\ast}_i &= y_n + H \sum_{j=1}^{i-1} \abase_{i,j} \, f_j + H \, \abase_{i,i} f^{\ast}_i, \\
				T_{i-1} &= t_n + \cbase_{i-1} \, H, \\
				v_i(0) &= Y_{i-1}, \\
				v_i' &= \deltac_{i} \, f\F \mleft(T_{i-1} + \deltac_{i} \, \theta, v_i \mright) + \sum_{j=1}^{i-1} \gamma_{i,j} \mleft( \tfrac{\theta}{H} \mright) \, f\S_j \\
				& \quad + \sum_{j=1}^{i} \psi_{i,j} \mleft( \tfrac{\theta}{H} \mright) \, f\S[\ast]_i, \quad \text{for } \theta \in [0, H], \\
				Y_{i} &= v_i(H), \qquad i = 1, \dots, s\S,
			\end{aligned} \right. \\
			\label{eqn:IPC-MRI-GARK_solution}
			& y_{n+1} = Y_{s\S},
		\end{align}
	\end{subequations}
	with $f\S[\ast]_j \coloneqq f\S \mleft( T_j, Y^{\ast}_j \mright)$ and $f^{\ast}_j \coloneqq f \mleft( T_j, Y^{\ast}_j \mright)$.  Stages and functions with an asterisk are predictor values, and terms without the asterisk are corrector values.  In order to enforce that only previously computed stages appear in the ODE, we require that $\gamma_{i,j}(\tau) = 0$ for $j \geq i$ and $\psi_{i,j}(\tau) = 0$ for $j>i$.
\end{definition}
Once again, we can take each $\gamma_{i,j}(t)$ and $\psi_{i,j}(t)$ to be polynomial in time.  These and their integral terms $\~\gamma_{i,j}(t)$, $\-\gamma_{i,j}$, $\~\psi_{i,j}(t)$, $\-\psi_{i,j}$ are defined analogously to \cref{eqn:gamma_series}.  The capitalized versions are used to denote the matrices of coefficients.

\begin{remark}[Embedded method]
	Following the strategy used in \cite{Sandu_2018_MRI-GARK}, an embedded solution can be obtained via the additional integration
	\begin{align*}
		\^v' &= \deltac_{s\S}  f\F \mleft(T_{s\S-1} + \deltac_{s\S} \, \theta, \^v \mright) + \sum_{j=1}^{s\S-1} \^\gamma_{j} \mleft( \tfrac{\theta}{H} \mright) \, f\S_j \\
		& \quad + \sum_{j=1}^{s\S} \^\psi_{j} \mleft( \tfrac{\theta}{H} \mright) \, f\S[\ast]_i, \quad \text{for } \theta \in [0, H], \\
		\^y_{n+1} &= \^v(H).
	\end{align*}
\end{remark}
With the trivial partitioning $f\S = f$, $f\F=0$, the corrector stages simplify to
\begin{equation}
	\label{eqn:IPC-MRI-GARK_ff0}
	\begin{split}
		Y_i &= Y_{i-1} + H \sum_{j=1}^{i-1} \-\gamma_{i,j} \, f\S_j + H \sum_{j=1}^{i} \-\psi_{i,j} \, f\S[\ast]_j \\
		&= y_n + H \sum_{j=1}^{i-1} \left( \sum_{\ell=j+1}^{i} \-{\gamma}_{\ell,j} \right) f\S_j + H \sum_{j=1}^{i} \left( \sum_{\ell=j}^{i} \-{\psi}_{\ell,j} \right) f\S[\ast]_j.
	\end{split}
\end{equation}
The \IPCabbv{} method becomes a $2 s\S$ stage Runge--Kutta method with $s\S$ predictor stages and $s\S$ corrector stages.  In the absence of the fast component, it is natural to expect the predictor and corrector stages to coincide and for the method to degenerate into the slow base scheme.  Matching coefficients of \cref{eqn:IPC-MRI-GARK_ff0} to those of the predictor stage of \cref{eqn:IPC-MRI-GARK} gives the self-consistency conditions
\begin{equation}
	\label{eqn:IPC-MRI-GARK_integral_condition}
	\Tbase = E \, \-\Gamma
	\quad \text{and} \quad
	\Dbase = E \, \-\Psi,
\end{equation}
where $\Tbase$ is the strictly lower triangular part of $\Abase$ and
\begin{align*}
	E &= \begin{bmatrix}
		1 & \dots & 0 \\
		\vdots & \ddots & \vdots \\
		1 & \dots & 1
	\end{bmatrix} \in \Re^{s\S \times s\S}, \qquad
	\Dbase = \diag{\left( \abase_{1,1}, \dots, \abase_{s\S,s\S} \right)}.
\end{align*}

\begin{remark}[Repeated abscissae]
	\label{rem:IPC-MRI-GARK_repeated_c}
	When $\cbase_i = \cbase_{i-1}$, the fast function disappears from the ODE in \cref{eqn:IPC-MRI-GARK_internal_ode} as it is scaled by zero.  The corrector stage simplifies to
	\begin{equation*}
	\begin{split}
\ifreport
		Y_i &= Y_{i-1} + \int_{0}^{H} \left( \sum_{j=1}^{i-1} \gamma_{i,j} \mleft( \tfrac{\theta}{H} \mright) \, f\S_j + \sum_{j=1}^{i} \psi_{i,j} \mleft( \tfrac{\theta}{H} \mright) \, f\S[\ast]_j \right) \dd{\theta} \\
\fi
		Y_i &= Y_{i-1} + H \sum_{j=1}^{i-1} \-\gamma_{i,j} \, f\S_j + H \sum_{j=1}^{i} \-\psi_{i,j} \, f\S[\ast]_j.
\end{split}
	\end{equation*}
	Clearly, an ODE solver is no longer needed to compute $Y_i$.  This can be viewed as modifying (the slow part of) the initial conditions for the next step's ODE.
\end{remark}


For component partitioned systems \cref{eqn:ode_component}, an \IPCabbv{} step reads:
\begin{subequations}
	\label{eqn:IPC-MRI-GARK_component}
	\methodstretch
	\begin{align}
		& Y\F_0 = y\F_n, \quad Y\S_0 = y\S_n, \quad \cbase_0 = 0, \\
		& \left\{ \begin{aligned}
			\begin{bmatrix} Y\F[\ast]_i \\ Y\S[\ast]_i \end{bmatrix}
			&= \begin{bmatrix}
				y\F_n + H \sum_{j=1}^{i-1} \abase_{i,j} \, f\F_{j} + H \abase_{i,i} \, f\F[\ast]_i \\
				y\S_n + H \sum_{j=1}^{i-1} \abase_{i,j} \, f\S_{j} + H \abase_{i,i} \, f\S[\ast]_i
			\end{bmatrix}, \\
			Y\S_i &= Y\S[\ast]_i \\
			v\F(0) &= Y\F_{i-1}, \\
			T_{i-1} &= t_n + \cbase_{i-1} \, H, \\
			{v_i\F}' &= \deltac_{i} f\F \mleft(T_{i-1} + \deltac_{i} \, \theta, v_i\F, Y\S_{i-1} + H \sum_{j=1}^{i} \~\delta_{i,j} \mleft( \tfrac{\theta}{H} \mright) \, f\S_j \mright), \\
			& \quad \text{for } \theta \in [0, H], \label{eqn:IPC-MRI-GARK_corrector} \\
			Y\F_i &= v_i\F(H), \quad i=1,\dots,s\S,
		\end{aligned} \right. \\
		%
		& \begin{bmatrix} y\F_{n+1} \\ y\S_{n+1} \end{bmatrix}
		= \begin{bmatrix} Y\F_{s\S} \\ Y\S_{s\S} \end{bmatrix},
	\end{align}
\end{subequations}
where $\~\delta_{i,j} \mleft( \tfrac{\theta}{H} \mright) = \~\gamma_{i,j} \mleft( \tfrac{\theta}{H} \mright) + \~\psi_{i,j} \mleft( \tfrac{\theta}{H} \mright)$ and $f\F[\ast]_j \coloneqq f\F \big( T_j, Y\F[\ast]_j, Y\S[\ast]_j \big)$.

\subsection{Order conditions}

Following \cref{sec:SPC-MRI-GARK_order_conditions}, we look to utilize GARK order condition theory to derive order conditions for \IPCabbv{} methods.  Again, we apply an arbitrarily accurate Runge--Kutta method $\left( A\FF, b\F, c\F \right)$ to discretize the ODEs and recover the GARK stages and GARK tableau.  We use the labels $\pre$ and $\cor$ to denote predictor and corrector stages, respectively.  Also we define $Y\FL[i]_k$ to be the $k$-th stage of the discretized ODE between abscissae $\cbase_{i - 1}$ and $\cbase_{i}$.  Now, the $i$-th step of \cref{eqn:IPC-MRI-GARK} is composed of the GARK stages
\begin{align*}
	Y\FL[\pre]_i &= Y\SL[\pre]_i
	= y_n + H \sum_{j=1}^{i-1} \abase_{i,j} \, f\FL[\cor]_j + H \, \abase_{i,i} \, f\FL[\pre]_i + H \sum_{j=1}^{i-1} \abase_{i,j} \, f\SL[\cor]_j \\
	& \quad + H \, \abase_{i,i} \, f\SL[\pre]_i, \\
	Y\FL[i]_k
	&= Y\FL[\cor]_{i-1} + H \sum_{j=1}^{s\F} a\FF_{k,j} \left( \deltac_i f\FL[i]_j + \sum_{\ell=1}^{i-1} \gamma_{i,\ell} \big( c\F_j \big) \, f\SL[\cor]_\ell \right. \\
	& \quad \left. + \sum_{\ell=1}^{i} \psi_{i,\ell} \big( c\F_j \big) \, f\SL[\pre]_{\ell} \right), \\
	Y\FL[\cor]_i &= Y\SL[\cor]_i
	= Y\FL[\cor]_{i-1} + H \sum_{j=1}^{s\F} b\F_{j} \left( \deltac_i f\FL[i]_j + \sum_{\ell=1}^{i-1} \gamma_{i,\ell} \big( c\F_j \big) \, f\SL_\ell \right. \\
	& \quad \left. + \sum_{\ell=1}^{i} \psi_{i,\ell} \big( c\F_j \big) \, f\SL[\pre]_{\ell} \right), \\
\end{align*}
with $f\FL[i]_j \coloneqq f\F \big( T_{i-1} + \deltac_i \, c\F_j, Y\FL[i]_j \big)$ and $f\comp{\sigma, \nu}_j \coloneqq f\comp{\sigma} \big( T_j, Y\comp{\sigma, \nu}_j \big)$ for $\sigma \in \{\fast, \slow\}$ and $\nu \in \{\pre, \cor\}$.  Now, we simplify $Y\FL[\cor]_i$ to obtain:
\begin{align*}
	Y\FL[\cor]_i &= Y\FL[\cor]_{i-1} + \deltac_i H \sum_{j=1}^{s\F} b\F_{j} \, f\FL[i]_j + H \sum_{j=1}^{s\F} b\F_{j} \sum_{\ell=1}^{i-1} \left( \sum_{k \geq 0} \gamma^k_{i,\ell} \, c\F[\times k]_{j} \right) f\SL[\cor]_\ell \\
	& \quad + H \sum_{j=1}^{s\F} b\F_{j} \sum_{\ell=1}^{i} \left( \sum_{k \geq 0} \psi^k_{i,\ell} \, c\F[\times k]_{j} \right) f\SL[\pre]_\ell \\
\ifreport
	&= Y\FL[\cor]_{i-1} + \deltac_i H \sum_{j=1}^{s\F} b\F_{j} \, f\FL[i]_j + H \sum_{\ell=1}^{i-1} \left( \sum_{k \geq 0} \frac{\gamma^k_{i,\ell}}{k+1} \right) f\SL[\cor]_\ell \\
	& \quad + H \sum_{\ell=1}^{i} \left( \sum_{k \geq 0} \frac{\psi^k_{i,\ell}}{k+1} \right) f\SL[\pre]_\ell \\
	&= Y\FL[\cor]_{i-1} + \deltac_i \, H \sum_{j=1}^{s\F} b\F_{j} f\FL[i]_j + H \sum_{j=1}^{i-1} \-\gamma_{i,j} \, f\SL[\cor]_j + H \sum_{j=1}^{i} \-\psi_{i,j} \, f\SL[\pre]_j \\
\fi
	&= y_n + H \sum_{j=1}^{i} \deltac_j \sum_{\ell=1}^{s\F} b\F_{\ell} \, f\FL[j]_{\ell} + H \sum_{j=1}^{i-1} \abase_{i,j} \, f\SL[\cor]_j + H \, \abase_{i,i} \, f\SL[\pre]_i.
\end{align*}
%
The stages of the discretized ODE simplify to

\begin{align*}
	Y\FL[i]_k &= y_n + H \sum_{j=1}^{i-1} \deltac_j \sum_{\ell=1}^{s\F} b\F_{\ell} \, f\FL[j]_{\ell} + \deltac_i H \sum_{j=1}^{s\F} a\FF_{k,j} \, f\FL[i]_j \\
	& \quad + H \sum_{j=1}^{i-2} \abase_{i,j} \, f\SL[\cor]_j + H \sum_{j=1}^{i-1} \left( \sum_{\ell=1}^{s\F} a\FF_{k,\ell} \, \gamma_{i,j} \big( c\F_j \big) \right) f\SL[\cor]_j \\
	& \quad + H \abase_{i-1,i-1} \, f\SL[\pre]_{i-1} + H \sum_{j=1}^{i} \left( \sum_{\ell=1}^{s\F} a\FF_{k,\ell} \, \psi_{i,j} \big( c\F_j \big) \right) f\SL[\pre]_j.
\end{align*}
The coefficients appearing in the stages can be organized into the following GARK tableau:
\begin{equation} \label{eqn:IPC-MRI-GARK_tableau}
	\begin{rktableau}{c|ccc|cc|c}
		\cbase & \Dbase & 0 & \Tbase & \Dbase & \Tbase & \cbase \\
		c\comp{\fast,\fast,i} & 0 & A\comp{\fast,\fast,i,i} & 0 & A\comp{\fast,\slow,i,\pre} & A\comp{\fast,\slow,i,\cor} & c\comp{\fast,\slow,i} \\
		\cbase & 0 & A\comp{\fast,\fast,\cor,i} & 0 & \Dbase & \Tbase & \cbase \\ \hline
		\cbase & \Dbase & 0 & \Tbase & \Dbase & \Tbase & \cbase \\
		\cbase & 0 & A\comp{\slow,\fast,\cor,i} & 0 & \Dbase & \Tbase & \cbase \\ \hline
		& 0 & \deltac* \otimes b\F* & 0 & e_{s\S}^T \Dbase & e_{s\S}^T \Tbase
	\end{rktableau}.
\end{equation}
The unspecified entries are
\begin{align*}
	A\comp{\fast,\fast,i,i} &= L \, \deltaC \otimes \one\F \, b\F* + \diag{\left( \deltac \right)} \otimes A\FF, \\
	A\comp{\fast,\slow,i,\pre} &= \sum_{k \geq 0} \Psi^k \otimes A\FF \, c\F[\times k] + L \, \Dbase \otimes \one\F, \\
	A\comp{\fast,\slow,i,\cor} &= \sum_{k \geq 0} \Gamma^k \otimes A\FF \, c\F[\times k] + L \, \Tbase \otimes \one\F, \\
	A\comp{\fast,\fast,\cor,i} &= A\comp{\slow,\fast,\cor,i} = \deltaC \otimes b\F*, \\
	c\comp{\fast,\fast,i} &= L \, \cbase \otimes \one\F + \deltac \otimes c\F, \\
	c\comp{\fast,\slow,i} &= L \, \cbase \otimes \one\F + \sum_{k \geq 0} \left( \Psi^k + \Gamma^k \right) \one\S \otimes A\FF \, c\F[\times k],
\end{align*}
with
\begin{align*}
	\deltaC &= \begin{bmatrix}
		\deltac_1 \\
		\deltac_1 & \deltac_2 \\
		\vdots & \vdots & \ddots \\
		\deltac_1 & \deltac_2 & \dots & \deltac_{s\S}
	\end{bmatrix},
\end{align*}
and $L \in \Re^{s\S \times s\S}$ is a lower shift matrix with entries $L_{i,j} = \delta_{i,j+1}$.

\subsubsection{Internal Consistency}

\begin{theorem}[Internal consistency conditions]
	An \IPCabbv{} method \cref{eqn:IPC-MRI-GARK} fulfills the ``internal consistency'' conditions
	\begin{equation*}
		\*c\SF = \*c\SS \equiv \cbase
		\quad \text{and} \quad
		\*c\FF = \*c\FS
	\end{equation*}
	for any fast method iff the following conditions hold:
	\begin{equation}
		\label{eqn:IPC-MRI-GARK_ic}
		\left( \Psi^0 + \Gamma^0 \right) \one\S = \deltac
		\quad \text{and} \quad
		\left( \Psi^k + \Gamma^k \right) \one\S = 0 \quad \forall k \geq 1.
	\end{equation}
\end{theorem}

\begin{proof}
	All internal consistency equations are automatically satisfied except
	\begin{align*}
		c\comp{\fast,\fast,i} &= c\comp{\fast,\slow,i} \qquad \Leftrightarrow \\
		L \cbase \otimes \one\F + \deltac \otimes c\F
		&= L \cbase \otimes \one\F + \sum_{k \geq 0} \left( \Psi^k + \Gamma^k \right) \one\S \otimes A\FF c\F[\times k].
	\end{align*}
It is easy to confirm \cref{eqn:IPC-MRI-GARK_ic} is sufficient to satisfy this condition, and thus, internal consistency.  Since the equality must hold for all $A\FF$, it must hold when all $A\FF c\F[\times k]$ are linearly independent.  Matching powers of the left- and right-hand sides proves the necessity of \cref{eqn:IPC-MRI-GARK_ic}.
\end{proof}

Like with \SPCabbv{} methods, internal consistency and a slow base method of order two guarantees an \IPCabbv{} method is order two \cite{Sandu_2015_GARK}.

\subsubsection{Fourth order conditions}
	
In this section, we derive order conditions of the \IPCabbv{} schemes for up to order four.

\ifreport
\begin{lemma}[Intermediate matrix products]
	The coefficients of the GARK tableau \cref{eqn:IPC-MRI-GARK_tableau} satisfy
	\begin{subequations}
		\methodstretch
		\begin{align}
			\label{eqn:IPC-MRI-GARK_Ass_Cs}
			\*A\SS \, \*c\S[\times \ell] &= \begin{bmatrix}
				\Abase \, \cbase[\times \ell] \\
				\Abase \, \cbase[\times \ell]
			\end{bmatrix}, \\
			\label{eqn:IPC-MRI-GARK_Asf_Cf}
			\*A\SF \, \*c\F[\times \ell] &= \begin{bmatrix}
				\Abase \, \cbase[\times \ell] \\
				\frac{1}{\ell + 1} \cbase[\times (\ell + 1)]
			\end{bmatrix}, \\
			\label{eqn:IPC-MRI-GARK_Afs_Cs}
			\*A\FS \, \*c\S[\times \ell] &= \begin{bmatrix}
				\Abase \, \cbase[\times \ell] \\
				L \, \Abase \, \cbase[\times \ell] \otimes \one\F + \sum_{k \geq 0} \left( \Psi^k + \Gamma^k \right) \cbase[\times \ell] \otimes A\FF \, c\F[\times k] \\
				\Abase \, \cbase[\times \ell]
			\end{bmatrix}.
		\end{align}
	\end{subequations}
\end{lemma}
\fi

\begin{theorem}[Fourth order coupling conditions]
	An internally consistent \IPCabbv{} method \cref{eqn:IPC-MRI-GARK} satisfying \cref{eqn:IPC-MRI-GARK_integral_condition} has order four iff the slow base scheme has order at least four, and the following coupling conditions hold:
	\begin{subequations}
		\label{eqn:IPC-MRI-GARK_oc}
		\allowdisplaybreaks
		\begin{align}
			\label{eqn:IPC-MRI-GARK_oc:3a}
			\frac{1}{6} &= \deltac* \left( L \, \Abase + \sum_{k \geq 0} \zeta_k \left( \Psi^k + \Gamma^k \right) \right) \cbase, & \text{(order 3)} \\
			\label{eqn:IPC-MRI-GARK_oc:3b}
			\frac{1}{6} &= e_{s\S}^T \left( \Dbase \, \Abase \, \cbase + \frac{1}{2} \Tbase \, \cbase[\times 2] \right), & \text{(order 3)} \\
			\label{eqn:IPC-MRI-GARK_oc:4a}
			\begin{split}
				\frac{1}{8} &= \left( \deltac \times L \, \cbase \right)^T \left( L \, \Abase + \sum_{k \geq 0} \zeta_k \left( \Psi^k + \Gamma^k \right) \right) \cbase \\
				& \quad + \left( \deltac[\times 2] \right)^T \left( \frac{1}{2} L \, \Abase + \sum_{k \geq 0} \psi_k \left( \Psi^k + \Gamma^k \right) \right) \cbase,
			\end{split} & \text{(order 4)} \\
			\label{eqn:IPC-MRI-GARK_oc:4b}
			\begin{split}
				\frac{1}{8} &= \left( e_{s\S}^T \, \Dbase \times \cbase* \right) \Abase \, \cbase \\
				& \quad + \frac{1}{2} \left( e_{s\S}^T \, \Tbase \times \cbase* \right) \cbase[\times 2]
			\end{split} & \text{(order 4)} \\
			\label{eqn:IPC-MRI-GARK_oc:4c}
			\frac{1}{12} &= \deltac* \left( L \, \Abase + \sum_{k \geq 0} \zeta_k \left( \Psi^k + \Gamma^k \right) \right) \cbase[\times 2], & \text{(order 4)} \\
			\label{eqn:IPC-MRI-GARK_oc:4d}
			\frac{1}{12} &= e_{s\S}^T \left( \Dbase \, \Abase \, \cbase[\times 2] + \frac{1}{3} \Tbase \, \cbase[\times 3] \right), & \text{(order 4)} \\
			\label{eqn:IPC-MRI-GARK_oc:4e}
			\begin{split}
				\frac{1}{24} &= \deltac* \, L \, \deltaC \left( L \, \Abase + \sum_{k \geq 0} \zeta_k \left( \Psi^k + \Gamma^k \right) \right) \cbase \\
				& \quad + \left( \deltac[\times 2] \right)^T \left( \frac{1}{2} L \, \Abase + \sum_{k \geq 0} \xi_k \left( \Psi^k + \Gamma^k \right) \right) \cbase,
			\end{split} & \text{(order 4)} \\
			\label{eqn:IPC-MRI-GARK_oc:4f}
			\begin{split}
				\frac{1}{24} &= \deltac* \left( L \, \Dbase + \sum_{k \geq 0} \zeta_k \Psi^k \right) \Abase \, \cbase \\
				& \quad + \frac{1}{2} \deltac* \left( L \, \Tbase + \sum_{k \geq 0} \zeta_k \Gamma^k \right) \cbase[\times 2],
			\end{split} & \text{(order 4)} \\
			\label{eqn:IPC-MRI-GARK_oc:4g}
			\frac{1}{24} &= \deltac* \left( L \, \Abase + \sum_{k \geq 0} \zeta_k \left( \Psi^k + \Gamma^k \right) \right) \Abase \, \cbase, & \text{(order 4)} \\
			\label{eqn:IPC-MRI-GARK_oc:4h}
			\frac{1}{24} &= e_{s\S}^T \, \Abase \left( \Dbase \, \Abase \, \cbase + \frac{1}{2} \Tbase \, \cbase[\times 2] \right), & \text{(order 4)} \\
			\label{eqn:IPC-MRI-GARK_oc:4i}
			\begin{split}
				\frac{1}{24} &= e_{s\S}^T \, \Dbase \, \Abase \, \Abase \, \cbase \\
				& \quad + e_{s\S}^T \, \Tbase \, \deltaC \left( L \, \Abase + \sum_{k \geq 0} \zeta_k \left( \Psi^k + \Gamma^k \right) \right) \cbase.
			\end{split} & \text{(order 4)}
		\end{align}
	\end{subequations}
	%
\end{theorem}

\begin{proof}
	An internally consistent GARK scheme is order four iff the base methods are order four and the 12 coupling conditions up to order four are satisfied \cite{Sandu_2015_GARK}. We proceed with checking each coupling condition.
	\paragraph{Condition 3a}
	\ifreport
	By using \cref{eqn:IPC-MRI-GARK_Afs_Cs}, the first third order condition gives \cref{eqn:IPC-MRI-GARK_oc:3a}:
	\fi
	\begin{align*}
		\frac{1}{6}
		&= \*b\F* \, \*A\FS \, \*c\S \\
		&= \left( \deltac \otimes b\F \right)^T \left( L \, \Abase \, \cbase \otimes \one\F + \sum_{k \geq 0} \left( \Psi^k + \Gamma^k \right) \cbase \otimes A\FF \, c\F[\times k] \right) \\
		&= \deltac* \left( L \, \Abase + \sum_{k \geq 0} \zeta_k \left( \Psi^k + \Gamma^k \right) \right) \cbase.
	\end{align*}
	\paragraph{Condition 3b}
	\ifreport
	The other third order condition is expanded with \cref{eqn:IPC-MRI-GARK_Asf_Cf} to get \cref{eqn:IPC-MRI-GARK_oc:3b}:
	\fi
	\begin{align*}
		\frac{1}{6}
		&= \*b\S* \, \*A\SF \, \*c\F
		= e_{s\S}^T \left( \Dbase \, \Abase \, \cbase + \frac{1}{2} \Tbase \, \cbase[\times 2] \right).
	\end{align*}
	\paragraph{Condition 4a}
	\ifreport
	By using \cref{eqn:IPC-MRI-GARK_Afs_Cs}, the first fourth order condition gives \cref{eqn:IPC-MRI-GARK_oc:4a}:
	\fi
	\begin{align*}
		\frac{1}{8}
		&= \left( \*b\F \times \*c\F \right)^T \*A\FS \, \*c\S \\
		&= \left( \left( \deltac \otimes b\F \right) \times \left( L \, \cbase \otimes \one\F + \deltac \otimes c\F \right) \right)^T \\
		& \quad \left( L \, \Abase \, \cbase \otimes \one\F + \sum_{k \geq 0} \left( \Psi^k + \Gamma^k \right) \cbase \otimes A\FF c\F[\times k] \right) \\
		&= \left( \deltac \times L \, \cbase \right)^T \left( L \, \Abase + \sum_{k \geq 0} \zeta_k \left( \Psi^k + \Gamma^k \right) \right) \cbase \\
		& \quad + \left( \deltac[\times 2] \right)^T \left( \frac{1}{2} L \, \Abase + \sum_{k \geq 0} \psi_k \left( \Psi^k + \Gamma^k \right) \right) \cbase.
	\end{align*}
	\paragraph{Condition 4b}
	\ifreport
	We derive \cref{eqn:IPC-MRI-GARK_oc:4b} with \cref{eqn:IPC-MRI-GARK_Asf_Cf}:
	\fi
	\begin{align*}
		\frac{1}{8}
		&= \left( \*b\S \times \*c\S \right)^T \*A\SF \, \*c\F \\
		&= \left( e_{s\S}^T \, \Dbase \times \cbase* \right) \Abase \, \cbase + \frac{1}{2} \left( e_{s\S}^T \, \Tbase \times \cbase* \right) \cbase[\times 2].
	\end{align*}
	\paragraph{Condition 4c}
	\ifreport
	We derive \cref{eqn:IPC-MRI-GARK_oc:4c} with \cref{eqn:IPC-MRI-GARK_Afs_Cs}:
	\fi
	\begin{align*}
		\frac{1}{12}
		&= \*b\F* \, \*A\FS \, \*c\S[\times 2]
		= \deltac* \left( L \, \Abase + \sum_{k \geq 0} \zeta_k \left( \Psi^k + \Gamma^k \right) \right) \cbase[\times 2].
	\end{align*}
	\paragraph{Condition 4d}
	\ifreport
	We derive \cref{eqn:IPC-MRI-GARK_oc:4d} with \cref{eqn:IPC-MRI-GARK_Asf_Cf}:
	\fi
	\begin{align*}
		\frac{1}{12}
		&= \*b\S* \, \*A\SF \, \*c\F[\times 2]
		= e_{s\S}^T \left( \Dbase \, \Abase \, \cbase[\times 2] + \frac{1}{3} \Tbase \, \cbase[\times 3] \right).
	\end{align*}
	\paragraph{Condition 4e}
	\ifreport
	We derive \cref{eqn:IPC-MRI-GARK_oc:4e} with \cref{eqn:IPC-MRI-GARK_Afs_Cs}:
	\fi
	\begin{align*}
		\frac{1}{24}
		&= \*b\F* \, \*A\FF \, \*A\FS \, \*c\S \\
		&= \left( \deltac \otimes b\F \right)^T \left( L \, \deltaC \otimes \one\F \, b\F* + \diag{\left( \deltac \right)} \otimes A\FF \right) \\
		& \quad \left( L \, \Abase \, \cbase \otimes \one\F + \sum_{k \geq 0} \left( \Psi^k + \Gamma^k \right) \cbase \otimes A\FF \, c\F[\times k] \right) \\
		&= \deltac* \, L \, \deltaC \left( L \, \Abase + \sum_{k \geq 0} \zeta_k \left( \Psi^k + \Gamma^k \right) \right) \cbase \\
		& \quad + \left( \deltac[\times 2] \right)^T \left( \frac{1}{2} L \, \Abase + \sum_{k \geq 0} \xi_k \left( \Psi^k + \Gamma^k \right) \right) \cbase.
	\end{align*}
	\paragraph{Condition 4f}
	\ifreport
	We derive \cref{eqn:IPC-MRI-GARK_oc:4f} with \cref{eqn:IPC-MRI-GARK_Asf_Cf}:
	\fi
	\begin{align*}
		\frac{1}{24}
		&= \*b\F* \, \*A\FS \, \*A\SF \, \*c\F \\
		&= \frac{1}{2} \left( \deltac \otimes b\F \right)^T \left( L \, \Tbase \otimes \one\F + \sum_{k \geq 0} \Gamma^k \otimes A\FF \, c\F[\times k] \right) \cbase[\times 2] \\
		& \quad + \left( \deltac \otimes b\F \right)^T \left( L \, \Dbase \otimes \one\F + \sum_{k \geq 0} \Psi^k \otimes A\FF \, c\F[\times k] \right) \Abase \, \cbase \\
		&= \deltac* \left( L \, \Dbase + \sum_{k \geq 0} \zeta_k \Psi^k \right) \Abase \, \cbase \\
		& \quad + \frac{1}{2} \deltac* \left( L \, \Tbase + \sum_{k \geq 0} \zeta_k \Gamma^k \right) \cbase[\times 2].
	\end{align*}
	\paragraph{Condition 4g}
	\ifreport
	We derive \cref{eqn:IPC-MRI-GARK_oc:4g} with \cref{eqn:IPC-MRI-GARK_Ass_Cs}:
	\fi
	\begin{align*}
		\frac{1}{24}
		&= \*b\F* \, \*A\FS \, \*A\SS \, \*c\S 
		= \deltac* \left( L \, \Abase + \sum_{k \geq 0} \zeta_k \left( \Psi^k + \Gamma^k \right) \right) \Abase \, \cbase.
	\end{align*}
	\paragraph{Condition 4h}
	\ifreport
	We derive \cref{eqn:IPC-MRI-GARK_oc:4h} with \cref{eqn:IPC-MRI-GARK_Asf_Cf}:
	\fi
	\begin{align*}
		\frac{1}{24}
		&= \*b\S* \, \*A\SS \, \*A\SF \, \*c\F
		= e_{s\S}^T \, \Abase \left( \Dbase \, \Abase \, \cbase + \frac{1}{2} \Tbase \, \cbase[\times 2] \right).
	\end{align*}
	\paragraph{Condition 4i}
	\ifreport
	We derive \cref{eqn:IPC-MRI-GARK_oc:4i} with \cref{eqn:IPC-MRI-GARK_Afs_Cs}:
	\fi
	\begin{align*}
		\frac{1}{24} &= \*b\S* \*A\SF  \*A\FS \*c\S \\
		&= e_{s\S}^T \, \Dbase \, \Abase \, \Abase \, \cbase + e_{s\S}^T \, \Tbase \left( \deltaC \otimes b\F* \right) \\
		& \quad \left( L \, \Abase \, \cbase \otimes \one\F + \sum_{k \geq 0} \left( \Psi^k + \Gamma^k \right) \cbase \otimes A\FF \, c\F[\times k] \right) \\
		&= e_{s\S}^T \, \Dbase \, \Abase \, \Abase \, \cbase \\
		& \quad + e_{s\S}^T \, \Tbase \, \deltaC \left( L \, \Abase + \sum_{k \geq 0} \zeta_k \left( \Psi^k + \Gamma^k \right) \right) \cbase.
	\end{align*}
	\paragraph{Condition 4j} This condition is equivalent to condition 4d since the fast base method has an arbitrarily large stage order, and thus, $\*A\FF \*c\F = \frac{1}{2} \*c\F[\times 2]$:
	\begin{align*}
		\frac{1}{24} = \*b\S* \, \*A\SF \, \*A\FF \, \*c\F
		\quad \Leftrightarrow \quad
		\frac{1}{12} = \*b\S* \, \*A\SF \, \*c\F[\times 2].
	\end{align*}
\end{proof}

\subsection{Linear stability analysis}
\label{sec:stability}
%

\subsubsection{Scalar stability analysis}

We revisit the scalar linear test problem \cref{eqn:scalar_test_problem} now for \IPCabbv{} methods.
\ifreport
The predictor stages in vector form are given by
\begin{equation} \label{eqn:linear_predicted_stages}
	\begin{split}
		Y^{\ast}
		&= y_n \, \one\S + z \, \Tbase \, Y + z \, \Dbase \, Y^{\ast} \\
		&= \left( \eye{s\S} - z \, \Dbase \right)^{-1} \left( y_n \, \one\S + z \, \Tbase \, Y \right).
	\end{split}
\end{equation}
The internal ODEs become
\begin{equation*}
	v' = \lambda\F \diag{\big( \deltac \big)} \, v + \lambda\S \, \Psi \mleft( \tfrac{\theta}{H} \mright) \, Y^{\ast} + \lambda\S \, \Gamma \mleft( \tfrac{\theta}{H} \mright) \, Y.
\end{equation*}
Integrating and substituting in the predicted stages \cref{eqn:linear_predicted_stages} gives
\begin{align*}
	Y &= v(H) \\
	&= \diag{\big( \varphi_0 \big( \deltac \, z\F \big) \big)} \, L \, Y + \varphi_0 \big( \deltac_1 z\F \big) \, y_n \, e_1 + z\S \, \mu \big( z\F \big) \, Y \\
	& \quad + z\S \, \nu \big( z\F \big) \, Y^{\ast} \\
	&= \diag{\big( \varphi_0 \big( \deltac \, z\F \big) \big)} \, L \, Y + \varphi_0 \big( \deltac_1 \, z\F \big) \, y_n \, e_1 + z\S \mu \big( z\F \big) \, Y \\
	& \quad + z\S \, \nu \big( z\F \big) \left( \eye{s\S} - z \, \Dbase \right)^{-1} \left( y_n \, \one\S + z \, \Tbase \, Y \right) \\
	&= \mathfrak{M}\big( z\F, z\S \big)^{-1} \\
	& \quad \left( \varphi_0 \big( \deltac_1 \, z\F \big) \, e_1 + z\S \, \nu \big( z\F \big) \left( \eye{s\S} - z \, \Dbase \right)^{-1} \one\S \right) y_n,
\end{align*}
\fi
Applying \cref{eqn:IPC-MRI-GARK} to \cref{eqn:scalar_test_problem} gives the scalar stability function
\begin{equation}
	\label{eqn:IPC-MRI-GARK_scalar_stability}
	\begin{split}
		R\big( z\F, z\S \big) &\coloneqq e_{s\S}^T \mathfrak{M}\big( z\F, z\S \big)^{-1} \\
		& \quad \left( \varphi_0 \big( \deltac_1 \, z\F \big) \, e_1 + z\S \, \nu \big( z\F \big) \left( \eye{s\S} - z \, \Dbase \right)^{-1} \one\S \right),
	\end{split}
\end{equation}
with
\begin{align*}
	\mu\big( z\F \big) &\coloneqq \sum_{k \geq 0} \diag{\big( \varphi_{k+1}\big( \deltac \, z\F \big) \big)} \, \Gamma^k, \\
	\nu\big( z\F \big) &\coloneqq \sum_{k \geq 0} \diag{\big( \varphi_{k+1}\big( \deltac \, z\S \big) \big)} \, \Psi^k, \\
	\mathfrak{M}\big( z\F, z\S \big) &\coloneqq \eye{s\S} - \diag{\big( \varphi_0 \big( \deltac \, z\F \big) \big)} \, L - z\S \mu \big( z\F \big) \\
	& \quad - z\S \, z \, \nu \big( z\F \big) \left( \eye{s\S} - z \, \Dbase \right)^{-1} \Tbase.
\end{align*}
It can be verified that \cref{eqn:SPC-MRI-GARK_zf_lim} still holds.  As the slow part of the test problem becomes infinitely stiff, however, we would like the stability function to be bounded.  One natural way to enforce this is by ensuring $\mathfrak{M}$, and thus $\mathfrak{M}^{-1}$, remains bounded in the limit.  The last two terms in $\mathfrak{M}$ are problematic since they are $\order{z\S}$.  If $\Abase$ is invertible, the following condition ensures these terms cancel in the limit:
\begin{equation*}
	\mu \big( z\F \big) = \nu \big( z\F \big) \left(\Dbase\right)^{-1} \Tbase.
\end{equation*}
Note that $\mu$ and $\nu$ are sums over linearly independent $\varphi$ functions.  By matching terms in this summation, we arrive at the stability simplifying assumption
\begin{equation}
	\label{eqn:IPC-MRI-GARK_stability_simplify}
	\Gamma^k = \Psi^k \left(\Dbase\right)^{-1} \Tbase, \qquad \forall k \ge 0.
\end{equation}
If $\Psi$ and $\Gamma$ are degree zero polynomials, then \cref{eqn:IPC-MRI-GARK_integral_condition} automatically ensures \cref{eqn:IPC-MRI-GARK_stability_simplify} is satisfied.

\subsubsection{Matrix stability analysis}
Now we consider the component partitioned \IPCabbv{} method \cref{eqn:IPC-MRI-GARK_component} applied to the matrix test problem \cref{eqn:matrix_test_problem}.  First, we define the following intermediate quantities:
\begin{align*}
	\mathfrak{P}_1(Z) &\coloneqq \left( \eye{2s\S} - Z \otimes \Dbase \right)^{-1} \left( \eye{2} \otimes \one\S \right), \\
	\mathfrak{P}_2(Z) &\coloneqq \left( \eye{2s\S} - Z \otimes \Dbase \right)^{-1} \left( Z \otimes \Tbase \right), \\
	\~\mu \big( z\F \big) &\coloneqq \sum_{k \geq 0} \diag{\left( \frac{\deltac \times \varphi_{k+2}\big( z\F \, \deltac \big)}{k+1} \right)} \Gamma^k, \\
	\~\nu \big( z\F \big) &\coloneqq \sum_{k \geq 0} \diag{\left( \frac{\deltac \times \varphi_{k+2}\big( z\F \, \deltac \big)}{k+1} \right)} \Psi^k.
\end{align*}
\ifreport
The predictor stages become
\begin{align*}
	\begin{bmatrix}
		Y\F[\ast] \\ Y\S[\ast]
	\end{bmatrix}
	&= \mathfrak{P}_1(Z) \begin{bmatrix}
		y\F_n \\
		y\S_n
	\end{bmatrix} + \mathfrak{P}_2(Z) \begin{bmatrix}
		Y\F \\
		Y\S
	\end{bmatrix}.
\end{align*}
The fast internal ODEs become
\begin{align*}
	{v\F}' &= \lambda\F \, \diag{\big( \deltac \big)} \, v\F + \eta\S \, \diag{\big( \deltac \big)} \, L \, Y\F \\
	&\quad + \eta\S \diag{\big( \deltac \big)} \, \~\Gamma\mleft( \tfrac{\theta}{H} \mright) \left( w\F \, Y\F + z\S \, Y\S \right) \\
	&\quad + \eta\S \diag{\big( \deltac \big)} \, \~\Psi\mleft( \tfrac{\theta}{H} \mright) \left( w\F \, y\F[\ast] + z\S \, y\S[\ast] \right).
\end{align*}
The solution to the system of ODEs gives the corrector stages
\begin{align*}
	Y\F &= v\F(H) \\
	&= \diag{\big( \varphi_0 \big( z\F \, \deltac \big) \big)} \left( L \, Y\F + y\F_n \, e_1 \right) \\
	& \quad + w\S \diag{\big( \deltac \times \varphi_1 \big( z\F \, \deltac \big) \big)} \left( L \, Y\S + y\S_n \, e_1 \right) \\
	& \quad + w\S \, \~\mu \big( z\F \big) \left( w\F \, Y\F + z\S \, Y\S \right) \\
	&\quad + w\S \, \~\nu \big( z\F \big) \left( w\F \, y\F[\ast] + z\S \, y\S[\ast] \right).
\end{align*}
The combined fast and slow corrector stages are
\begin{align*}
	\begin{bmatrix} Y\F \\ Y\S \end{bmatrix}
	&= \begin{bmatrix} 
		\diag{\big( \varphi_0 \big( z\F \, \deltac \big) \big)} \, L &  w\S \diag{\big( \deltac \times \varphi_1 \big( z\F \, \deltac \big) \big)} \, L \\
		0 & 0
	\end{bmatrix} \begin{bmatrix} Y\F \\ Y\S \end{bmatrix} \\
	& \quad + \begin{bmatrix}
		\varphi_0 \big( z\F \, \deltac_1 \big) \, e_1 & w\S \deltac_1 \varphi_1 \mleft( z\F \deltac_1 \mright) e_1 \\
		0 & 0
	\end{bmatrix} \begin{bmatrix} y\F_n \\ y\S_n \end{bmatrix} \\
	& \quad + \begin{bmatrix}
		w\S \, w\F \, \~\mu \big( z\F \big) & w\S \, z\S \, \~\mu \big( z\F \big) \\
		0 & 0
	\end{bmatrix} \begin{bmatrix} Y\F \\ Y\S \end{bmatrix} \\
	& \quad + \begin{bmatrix}
		w\S \, w\F \, \~\nu \big( z\F \big) & w\S \, z\S \, \~\nu \big( z\F \big) \\
		0 & \eye{s\S}
	\end{bmatrix} \begin{bmatrix} y\F[\ast] \\ y\S[\ast] \end{bmatrix} \\
	&= \mathfrak{N}_1(Z)^{-1} \mathfrak{N}_2(Z) \begin{bmatrix} y\F_n \\ y\S_n \end{bmatrix}.
\end{align*}
\fi
The stability matrix is given by
\begin{equation}
	\label{eqn:IPC-MRI-GARK_matrix_stability}
	\*M(Z) \coloneqq \left( \eye{2} \otimes e_{s\S}^T \right) \mathfrak{N}_1(Z)^{-1} \, \mathfrak{N}_2(Z),
\end{equation}
with
\begin{align*}
	\mathfrak{N}_1(Z) &= - \begin{bmatrix}
		\diag{\big( \varphi_0 \big( z\F \, \deltac \big) \big)} \, L & w\S \diag{\big( \deltac \times \varphi_1 \big( z\F \, \deltac \big) \big)} \, L \\
		0 & 0
	\end{bmatrix} \\
	& \quad - \begin{bmatrix}
		w\S \, w\F \, \~\mu \big( z\F \big) & w\S \, z\S \, \~\mu \big( z\F \big) \\
		0 & 0
	\end{bmatrix} \\
	& \quad - \begin{bmatrix}
		w\S \, w\F \, \~\nu \big( z\F \big) & w\S \, z\S \, \~\nu \big( z\F \big) \\
		0 & \eye{s\S}
	\end{bmatrix} \mathfrak{P}_2 + \eye{2 s\S} ,\\
	\mathfrak{N}_2(Z) &= \begin{bmatrix}
		\varphi_0 \big( z\F \, \deltac_1 \big) \, e_1 & w\S \, \deltac_1 \, \varphi_1 \big( z\F \, \deltac_1 \big) \, e_1 \\
		0 & 0
	\end{bmatrix} \\
	& \quad + \begin{bmatrix}
		w\S \, w\F \, \~\nu \big( z\F \big) & w\S \, z\S \, \~\nu \big( z\F \big) \\
		0 & \eye{s\S}
	\end{bmatrix} \mathfrak{P}_1.
\end{align*}

\subsection{Construction of practical methods}

We develop new implicit \IPCabbv{} methods up to order four which are presented in \cref{sec:IPC-MRI-GARK_methods}.  The second order base methods are reused from \SPCabbv{}, but the third and fourth order base methods are custom due to the nondecreasing abscissae constraint.  Upon deriving a parameterized family of $\Gamma$ and $\Psi$ coefficients that satisfy the coupling order conditions, we use free coefficients to satisfy the stability simplifying assumption \cref{eqn:IPC-MRI-GARK_stability_simplify}.  Any remaining parameters are used to optimize the size of the stability region.  Plots of the scalar and matrix stability regions are provided in \cref{fig:IPC-MRI-GARK_scalar_stability,fig:IPC-MRI-GARK_matrix_stability}, respectively.  Compared to \SPCabbv{} methods, we found it significantly more challenging to achieve large stability regions at high orders.
\section{Numerical results}
\label{sec:numerics}

In this section, we present the numerical tests performed on the \SPCabbv{} and \IPCabbv{} methods.

\subsection{Additive partitioning: the Gray--Scott model}

The first test problem considered is the Gray--Scott reaction-diffusion PDE \cite{Pearson189}:
\begin{equation}
	\label{eqn:Gray-Scott}
	\underbrace{
		\begin{bmatrix} u \\ v \end{bmatrix}'
	}_{y'}
	=
	\underbrace{
		\begin{bmatrix} \nabla \cdot ( \varepsilon_u \, \nabla u ) \\ \nabla \cdot ( \varepsilon_v \, \nabla v ) \end{bmatrix}
	}_{f\S(y)}
	+
	\underbrace{
		\begin{bmatrix} -u \, v^2 + \mathfrak{f} \, (1-u) \\  u \, v^2 - (\mathfrak{f} + \mathfrak{k}) \, v \end{bmatrix}
	}_{f\F(y)}.
\end{equation}
It is solved over the 2D spatial domain $[0,1] \times [0,1]$, which is discretized with second order finite differences. The timespan is taken to be $[0,30]$, and the model parameters are $\varepsilon_u = 0.0625$, $\varepsilon_v = 0.0312$, $\mathfrak{k}= 0.0520$, and $\mathfrak{f}= 0.0180$. 
The linear diffusions terms of \cref{eqn:Gray-Scott} make up the slow partition while the nonlinear reaction terms make up the fast partition.

MATLAB is used to carry out the convergence experiments. ODEs that appear within the integrators are solved using \texttt{ode45} with the tolerances \texttt{abstol = reltol = 1e-10}.
Convergence diagrams for the new methods presented in \cref{sec:new_methods} are shown in \cref{fig:Gray-Scott}. The numerical orders of accuracy are consistent with theoretical orders.

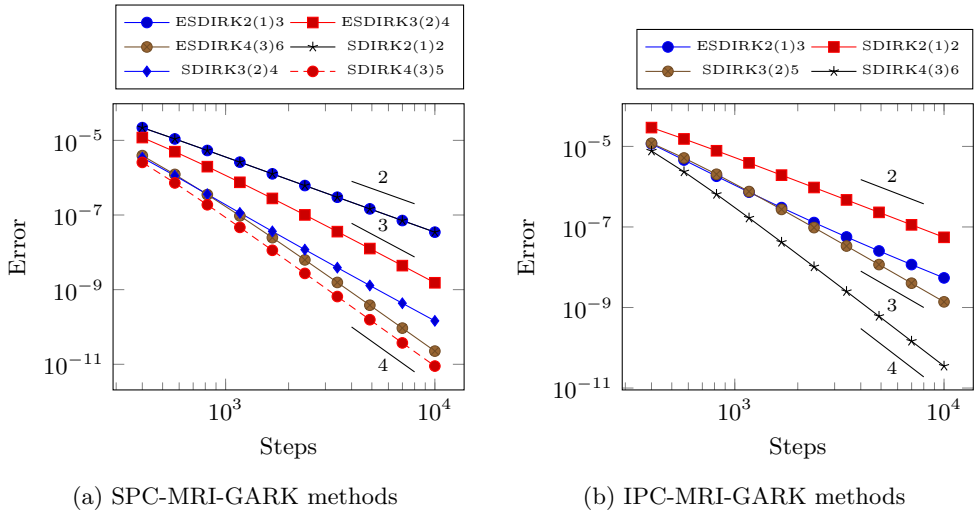
\begin{figure}[ht]
	\centering
	\begin{subfigure}[b]{.48\linewidth}
		\begin{tikzpicture}
			\begin{loglogaxis}[xlabel={Steps},ylabel={Error},legend columns=2,legend style={at={(1,1.05)}, anchor=south east}]
				\orderplot{Gray-Scott/SPC.dat}
				\draw (axis cs:4e3,8e-7) -- node[above]{\scriptsize $2$} (axis cs:8e3,2e-7);
				\draw (axis cs:4e3,6e-8) -- node[above]{\scriptsize $3$} (axis cs:8e3,7.5e-9);
				\draw (axis cs:4e3,1e-10) -- node[below]{\scriptsize $4$} (axis cs:8e3,6.25e-12);
			\end{loglogaxis}
		\end{tikzpicture}
		\caption{\SPCabbv{} methods}
	\end{subfigure}
	\hfill
	\begin{subfigure}[b]{.48\linewidth}
		\begin{tikzpicture}
			\begin{loglogaxis}[xlabel={Steps},ylabel={Error},legend columns=2,legend style={at={(1,1.05)}, anchor=south east}]
				\orderplot{Gray-Scott/IPC.dat}
				\draw (axis cs:4e3,1.5e-6) -- node[above]{\scriptsize $2$} (axis cs:8e3,3.75e-7);
				\draw (axis cs:4e3,8e-9) -- node[below]{\scriptsize $3$} (axis cs:8e3,1e-9);
				\draw (axis cs:4e3,3e-10) -- node[below]{\scriptsize $4$} (axis cs:8e3,1.89e-11);
			\end{loglogaxis}
		\end{tikzpicture}
		\caption{\IPCabbv{} methods}
	\end{subfigure}

	\caption{Error vs. number of steps for the Gray--Scott problem \cref{eqn:Gray-Scott}. Reference lines are used to indicate orders.}
	\label{fig:Gray-Scott}
\end{figure}

\subsection{Component partitioning: the KPR problem}
For a component partitioned test problem of the form \cref{eqn:ode_component}, we use the KPR system \cite{Sandu_2013_extrapolatedMR} as a multi-scale extension to the scalar Prothero-Robinson \cite{Bartel_2002_MR-W,Hairer_book_II,Prothero_1974_PR} problem. We define the system as:
\begin{subequations}
	\label{eqn:PR:Nonlinear}
	\begin{equation}
		\label{eqn:PR:Nonlinear:Equation}
		\begin{bmatrix}
			y\F \\
			y\S
		\end{bmatrix}'
		= \boldsymbol{\Omega}
		\cdot
		\begin{bmatrix}
			\frac{-3+y\F[\times 2] - \cos(\omega \, t)}{2 \, y\F} \\[4pt]
			\frac{-2+y\S[\times 2] - \cos(t)}{2 \, y\S}
		\end{bmatrix}
		- 
		\begin{bmatrix}
			\frac{\omega \sin(\omega \, t)}{2 \, y\F} \\[4pt]
			\frac{\sin(t)}{2 \, y\S}
		\end{bmatrix}.
	\end{equation}
	The parameters are chosen as $\lambda\F = -10$, $\lambda\S = -1$, $\xi = 0.1$, $\alpha = 1$, and $\omega = 20$.
	The exact solution of \cref{eqn:PR:Nonlinear:Equation} is given by:
	\begin{equation}
		\label{eqn:PR:Nonlinear:Exact:Solution}
		y\F(t) = \sqrt{3+\cos(\omega \, t)}, \qquad
		y\S(t) = \sqrt{2+\cos(t)}.
	\end{equation}
\end{subequations}
The tests are performed from $t = 0$ to $t = 5 \pi / 2$ with the initial condition coming from evaluating \cref{eqn:PR:Nonlinear:Exact:Solution} at $t = 0$.  From the exact solution we can also see that the differences in the fast and slow time scales are driven by $\omega$ and not $\lambda\F$ and $\lambda\S$.

The fast integration \cref{eqn:IPC-MRI-GARK_corrector} is also carried out using \texttt{ode45} solver with \texttt{abstol = reltol = 1e-10}. The convergence diagrams reported in \cref{fig:KPR} indicate that the methods perform at their theoretical orders for this problem. 

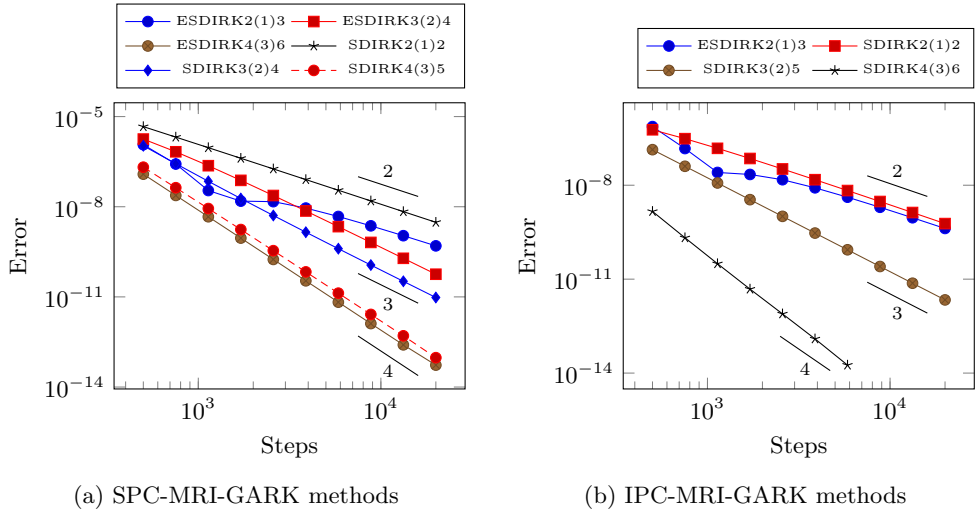
\begin{figure}[ht]
	\centering
	\begin{subfigure}[b]{.48\linewidth}
		\begin{tikzpicture}
			\begin{loglogaxis}[xlabel={Steps},ylabel={Error},legend columns=2,legend style={at={(1,1.05)}, anchor=south east}]
				\orderplot{KPR/SPC.dat}
				\draw (axis cs:7.5e3,1e-7) -- node[above]{\scriptsize $2$} (axis cs:1.6e4,2.2e-8);
				\draw (axis cs:7.5e3,6e-11) -- node[below]{\scriptsize $3$} (axis cs:1.6e4,6.18e-12);
				\draw (axis cs:7.5e3,5e-13) -- node[below]{\scriptsize $4$} (axis cs:1.6e4,2.41e-14);
			\end{loglogaxis}
		\end{tikzpicture}
		\caption{\SPCabbv{} methods}
	\end{subfigure}
	\hfill
	\begin{subfigure}[b]{.48\linewidth}
		\begin{tikzpicture}
			\begin{loglogaxis}[xlabel={Steps},ylabel={Error},legend columns=2,legend style={at={(1,1.05)}, anchor=south east}]
				\orderplot{KPR/IPC.dat}
				\draw (axis cs:7.5e3,2e-8) -- node[above]{\scriptsize $2$} (axis cs:1.6e4,4.39e-9);
				\draw (axis cs:7.5e3,8e-12) -- node[below]{\scriptsize $3$} (axis cs:1.6e4,8.24e-13);
				\draw (axis cs:2.5e3,1.5e-13) -- node[below]{\scriptsize $4$} (axis cs:4.7e3,1.2e-14);
			\end{loglogaxis}
		\end{tikzpicture}
		\caption{\IPCabbv{} methods}
	\end{subfigure}
	
	\caption{Error vs. number of steps for the KPR problem \cref{eqn:PR:Nonlinear}. Reference lines are used to indicate orders.}
	\label{fig:KPR}
\end{figure}

\subsection{Multirate performance: the inverter chain problem}

We also consider the inverter chain model of \cite{Kvaerno_1999_MR-RK} given by the equations
\begin{equation} \label{eqn:inverter_chain}
	\begin{split}
		U'_1 &= U_{op} - U_1 - \Gamma \, g(U_{in}, U_1, U_0), \\
		U'_i &= U_{op} - U_i - \Gamma \, g(U_{i-1}, U_i, U_0), \qquad i = 2, \dots m,
	\end{split}
\end{equation}
with $U_0 = 0$, $U_{op} = 5$, $U_T = 1$, $\Gamma = 100$, and
\begin{align*}
	g(U_G, U_D, U_S) = (\max(U_G - U_S - U_T, 0))^2 - (\max(U_G - U_D - U_T, 0))^2.
\end{align*}
The initial conditions of the system are
\begin{align*}
	U_i(0) = \begin{cases}
		6.246 \times 10^{-3} & i \text{ even} \\
		5 & i \text{ odd}
	\end{cases},
\end{align*}
and the input signal is taken to be
\begin{align*}
	U_{in}(t) = \begin{cases}
		t - 5 & 5 \le t \le 10 \\
		5 & 10 \le t \le 15 \\
		\frac{5}{2} (17 - t) & 15 \le t \le 17 \\
		0 & \text{otherwise}
	\end{cases}.
\end{align*}
For the numerical experiments, we use $m = 500$ and a timespan of $[0, 100]$.  As the signal propagates through the circuit, only a small percentage of the inverters experience a change in voltage while the other inverters maintain a constant voltage.  A componentwise partitioning of \cref{eqn:inverter_chain} is used where the fast components come from a sliding window that follows the signal, and the remaining components form the slow partition.

A C implementation of \cref{eqn:inverter_chain} is used to measure the performance gains provided by \SPCabbv{} and \IPCabbv{} over a single rate base method of the same order.  For order two, we compare SPC SDIRK2(1)2 from \cref{sec:SPC-MRI-GARK_methods:SDIRK2(1)2} and IPC SDIRK2(1)2 from \cref{sec:IPC-MRI-GARK_methods:SDIRK2(1)2} to their shared base method SDIRK2(1)2.  The results are plotted in \cref{fig:inverter_chain_work_precision:2}.  \Cref{fig:inverter_chain_work_precision:3} compares SPC SDIRK3(2)4 from \cref{sec:SPC-MRI-GARK_methods:SDIRK3(2)4} to its base method SDIRK3(2)4 and to IPC SDIRK3(2)5 from \cref{sec:IPC-MRI-GARK_methods:SDIRK3(2)5}.  Finally, \cref{fig:inverter_chain_work_precision:4} compares SPC ESDIRK4(3)6 from \cref{sec:SPC-MRI-GARK_methods:ESDIRK4(3)6} to its base method ESDIRK4(3)6 and to SPC SDIRK4(3)5 from \cref{sec:SPC-MRI-GARK_methods:SDIRK4(3)5}.  We did not include results for IPC SDIRK4(3)6 as it was only stable for timesteps much smaller than those used for the SPC methods.  This observation is consistent with the stability regions presented in \cref{fig:IPC-MRI-GARK_scalar_stability,fig:IPC-MRI-GARK_matrix_stability}.

Fixed timesteps were used for the coupled MRI-GARK methods, as well as for the method ESDIRK5(4)7[2]SA\textsubscript{2} from \cite{KENNEDY2019221} used to solve the internal ODEs.  In the experiments, ten timesteps were taken to solve these internal ODEs, except for IPC SDIRK2(1)2 and SPC SDIRK3(2)4 where five and 15 steps were used, respectively.

\begin{figure}[ht]
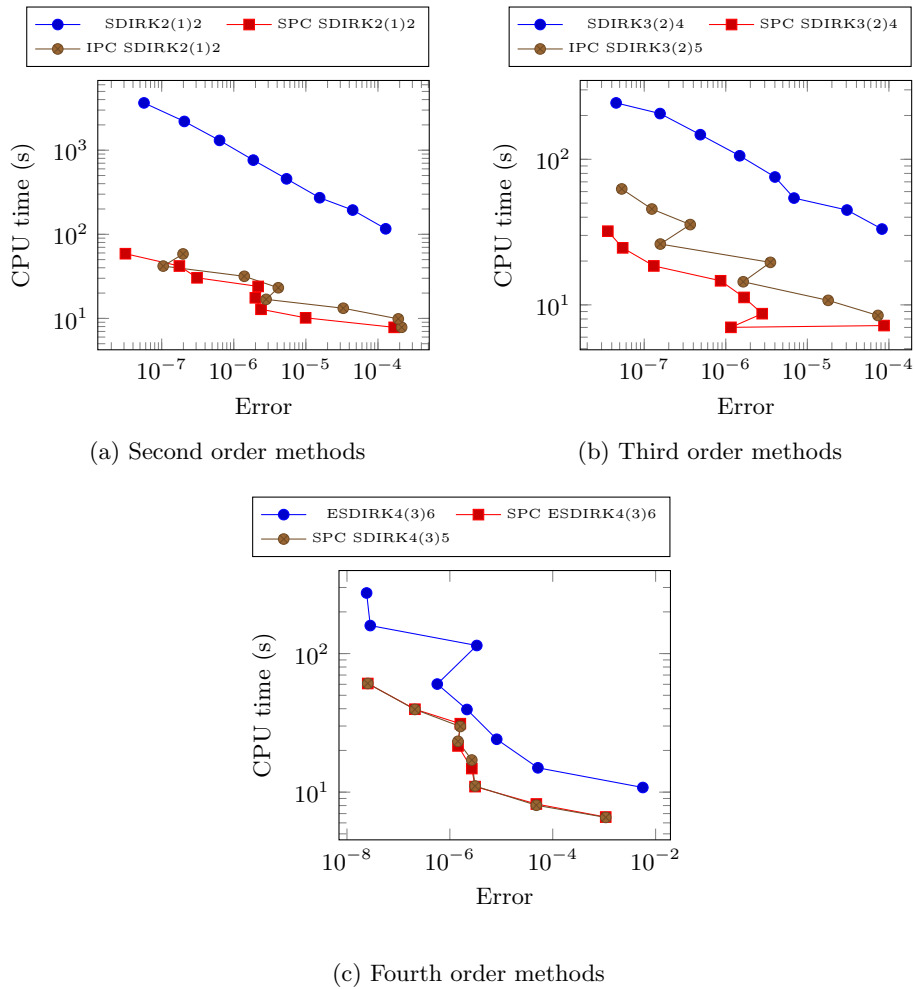

	\centering
	\begin{subfigure}[t]{0.46\linewidth}
		\perfplot{Inverter-Chain/order_2.dat}{SDIRK2(1)2,SPC SDIRK2(1)2,IPC SDIRK2(1)2}
		\caption{Second order methods}
		\label{fig:inverter_chain_work_precision:2}
	\end{subfigure}
	\hfil
	\begin{subfigure}[t]{0.46\linewidth}
		\perfplot{Inverter-Chain/order_3.dat}{SDIRK3(2)4,SPC SDIRK3(2)4,IPC SDIRK3(2)5}
		\caption{Third order methods}
		\label{fig:inverter_chain_work_precision:3}
	\end{subfigure}
	\\
	\begin{subfigure}[t]{0.46\linewidth}
		\perfplot{Inverter-Chain/order_4.dat}{ESDIRK4(3)6,SPC ESDIRK4(3)6,SPC SDIRK4(3)5}
		\caption{Fourth order methods}
		\label{fig:inverter_chain_work_precision:4}
	\end{subfigure}
	\caption{Work precision diagrams for single rate, \SPCabbv{}, and \IPCabbv{} methods applied to the inverter chain problem \cref{eqn:inverter_chain}.}
	\label{fig:inverter_chain_work_precision}
\end{figure}

In all of the performance results presented in \cref{fig:inverter_chain_work_precision}, the multirate methods are able to achieve a desired accuracy in significantly less time than the single rate schemes.  The best results occurred at order two where the speedup ranged from 8 to 60.  This can be attributed to the excellent multirate characteristics of the inverter chain problem as well as the flexibility of mutirate infinitesimal methods to use any method to solve the modified fast ODEs.
\section{Conclusions and future work}
\label{sec:conclusions}

This work extends the class of multirate infinitesimal GARK schemes developed in \cite{Sandu_2018_MRI-GARK} to include coupled methods. Such methods compute (some of) the stages by solving implicit systems that involve both the fast and the slow components, which gives their ``coupled'' character. The coupled approach allows us to construct multirate infinitesimal schemes with improved stability for stiff systems with multiple scales, at the additional cost of solving more complex, or larger, nonlinear systems.

Two approaches to formulating the coupling are studied herein. Both of them employ a predictor-corrector structure.
The first approach, named \SPCMethod{}, starts with computing all predictor stages in a coupled fashion. The predicted stages are then used to formulate a modified fast ODE, and a single infinitesimal integration is carried out to correct the fast component of the system. 
The second approach, named \IPCMethod{}, alternates prediction and correction stages. Specifically, each discrete predictor stage is followed by a corrector stage, which integrates a modified fast ODE system and corrects the fast components of that stage.

Elegant formulations of the order conditions for both families of methods are developed, and stability requirements for practical methods are analyzed. Methods of order up to four are constructed. Numerical tests verify the orders of convergence on additive and component partitioned cases. Finally, we demonstrate computational efficiency of these multirate methods when compared to their single rate counterparts. Our numerical experiments indicate a performance edge for both MRI-GARK strategies compared to single rate ones, with \SPCabbv{} methods having slightly better performance in general. Our analysis also shows larger stability regions for \SPCabbv{} methods compared to \IPCabbv{}.  

The succession of discrete and infinitesimal integration stages in the \IPCabbv{} schemes requires non-decreasing abscissae for the slow base method. This requirement, in conjunction with stability, can become difficult to satisfy for high order methods. One solution to alleviate this restriction is to construct MRI-GARK methods that compute their own initial conditions for the infinitesimal stage integration. The authors plan to study these extensions in future works.

\bibliographystyle{siamplain}
\bibliography{Bib/ode_general,Bib/ode_multirate,Bib/sandu,Bib/misc}

\appendix
\section{New MRI-GARK methods}
\label{sec:new_methods}

Here, we present the newly derived \SPCabbv{} and \IPCabbv{} methods.  In some cases, the exact representation of the method coefficients is too long to fit on a page, so the first 16 digits are provided.  Exact coefficients are available in the supplementary materials.  The stability regions are computed according to \cite{Sandu_2018_MRI-GARK}:
\begin{align*}
	\scalarstab &= \big\{ z\S \in \mathbb{C} \, \big| \, \big|R \big( z\F, z\S \big)\big| \le 1, \, \forall z\F \in \mathbb{C}^- : \big| z\F \big| \le \rho, \big|\text{arg}{\big( z\F \big)} - \pi \big| \le \alpha \big\}, \\
	\matstab &= \big\{ z\S \in \mathbb{C} \, \big| \, \max \big| \text{eig} \, \*M\big( z\F, z\F \big) \big| \le 1, \\
	& \qquad \forall z\F \in \mathbb{C}^- : \big| z\F \big| \le \rho, \big|\text{arg}{\big( z\F \big)} - \pi \big| \le \alpha \big\}.
\end{align*}

\subsection{\SPCabbv{} methods}
\label{sec:SPC-MRI-GARK_methods}

We use the following tableau representation for SPC-MRI-GARK methods:
\begin{equation*}
	\begin{rktableau}{c|ccc}
		\cbase_1 & \abase_{1,1} & \ldots & \abase_{1,s\S} \\
		\vdots & \vdots & \ddots & \vdots \\
		\cbase_{s\S} & \abase_{s\S,1} & \ldots & \abase_{s\S,s\S} \\ \hline
		& \gamma_1(t) & \ldots & \gamma_{s\S}(t) \\ \hline
		& \^\gamma_1(t) & \ldots & \^\gamma_{s\S}(t)
	\end{rktableau}.
\end{equation*}

\subsubsection{SDIRK2(1)2}
\label{sec:SPC-MRI-GARK_methods:SDIRK2(1)2}
This method is based on the two stage, second order method in \cite{Alexander_1977_SDIRK}.
\begin{equation*}
	\begin{rktableau}{c|cc}
		{\scriptstyle 1-}\frac{1}{\sqrt{2}} & {\scriptstyle 1-}\frac{1}{\sqrt{2}} & {\scriptstyle 0} \\
		{\scriptstyle 1} & \frac{1}{\sqrt{2}} & {\scriptstyle 1-}\frac{1}{\sqrt{2}} \\
		\hline
		\text{} & {\scriptstyle \left(12-9 \sqrt{2}\right) t+5 \sqrt{2}-6} & {\scriptstyle \left(9 \sqrt{2}-12\right) t-5 \sqrt{2}+7} \\
		\hline
		\text{} & \left(\frac{78}{5} {\scriptstyle - 12 \sqrt{2}}\right) {\scriptstyle t+6 \sqrt{2} -} \frac{36}{5} & \left({\scriptstyle 12 \sqrt{2} -} \frac{78}{5}\right) {\scriptstyle t-6 \sqrt{2} +}\frac{41}{5} \\
	\end{rktableau}
\end{equation*}

\subsubsection{ESDIRK2(1)3}
\label{sec:SPC-MRI-GARK_methods:ESDIRK2(1)3}
This method is based on TR-BDF2 in \cite{bank1985transient}.
\begin{equation*}
	\begin{rktableau}{c|ccc}
			{\scriptstyle 0} & {\scriptstyle 0} & {\scriptstyle 0} & {\scriptstyle 0} \\
			{\scriptstyle 2-\sqrt{2}} & {\scriptstyle 1-}\frac{1}{\sqrt{2}} & {\scriptstyle 1-}\frac{1}{\sqrt{2}} & {\scriptstyle 0} \\
			{\scriptstyle 1} & \frac{1}{2 \sqrt{2}} & \frac{1}{2 \sqrt{2}} & {\scriptstyle 1-}\frac{1}{\sqrt{2}} \\
			\hline
			\text{} & \left({\scriptstyle 6-}\frac{9}{\sqrt{2}}\right) {\scriptstyle t+}\frac{5}{\sqrt{2}}{\scriptstyle -3} & \left({\scriptstyle 6-}\frac{9}{\sqrt{2}}\right) {\scriptstyle t+}\frac{5}{\sqrt{2}}{\scriptstyle -3} & {\scriptstyle \left(9 \sqrt{2}-12\right) t-5 \sqrt{2}+7} \\
			\hline
			\text{} & \left(\frac{39}{5}{\scriptstyle -6 \sqrt{2}}\right) {\scriptstyle t +3 \sqrt{2}-}\frac{18}{5} & \left(\frac{39}{5}{\scriptstyle -6 \sqrt{2}}\right) {\scriptstyle t +3 \sqrt{2}-}\frac{18}{5} & \left({\scriptstyle 12 \sqrt{2}-}\frac{78}{5}\right) {\scriptstyle t-6 \sqrt{2}+}\frac{41}{5}
	\end{rktableau}
\end{equation*}

\subsubsection{SDIRK3(2)4}
\label{sec:SPC-MRI-GARK_methods:SDIRK3(2)4}
This method is based on SDIRK3M in \cite{Kennedy_2016_SDIRK-review}.
\begin{equation*}
	\begin{rktableau}{c|cccc}
		\frac{9}{40} & \frac{9}{40} & {\scriptstyle 0} & {\scriptstyle 0} & {\scriptstyle 0} \\
		\frac{7}{13} & \frac{163}{520} & \frac{9}{40} & {\scriptstyle 0} & {\scriptstyle 0} \\
		\frac{11}{15} & -\frac{6481433}{8838675} & \frac{87795409}{70709400} & \frac{9}{40} & {\scriptstyle 0} \\
		{\scriptstyle 1} & \frac{4032}{9943} & \frac{6929}{15485} & -\frac{723}{9272} & \frac{9}{40} \\
		\hline
		\text{} & -\frac{21765 t}{9943} & \frac{18740344238109 t}{12407262101200} & -\frac{2318739807 t}{928641703280} & \frac{341049771 t}{500777450} \\ [-0.2em]
		& +\frac{3}{2} & -\frac{46850957023}{152236344800} & -\frac{2336165553}{30447268960} & -\frac{231399837}{2003109800} \\
		\hline
		\text{} & -\frac{458 t}{153} & \frac{1143703567597 t}{484654507050} & \frac{12128361703356241349 t}{41321158297274157120} & \frac{6985915649614123877 t}{20539757048352651200} \\ [-0.2em]
		& +\frac{17}{9} & -\frac{5}{7} & -\frac{3214490524810792571}{14788625074813908864} & +\frac{70261070970241507}{1643180563868212096}
	\end{rktableau}
\end{equation*}

\subsubsection{ESDIRK3(2)4}
\label{sec:SPC-MRI-GARK_methods:ESDIRK3(2)4}
This method is based on the optimal four stage, third order ESDIRK method described in \cite{Kennedy_2016_SDIRK-review}.
\begin{equation*}
	\resizebox{\linewidth}{!}{$
		\begin{rktableau}{c|cccc}
			0 & 0 & 0 & 0 & 0 \\
			0.8717330430169180 & 0.4358665215084590 & 0.4358665215084590 & 0 & 0 \\
			0.6089666303771147 & 0.2648804871412033 & -0.09178037827254760 & 0.4358665215084590 & 0 \\
			1.000000000000000 & 0.1921013555637903 & -0.6181218831132021 & 0.9901540060409528 & 0.4358665215084590 \\
			\hline
			\text{} & 0.2335954530133717 t & 3.847836453450424 t & -4.416875540651942 t & 0.24354436341881466 t \\ [-0.5em]
			& +0.07530362905710443 & -2.542040109838414 & +3.198591776366924 & +0.2681447044143857 \\
			\hline
			\text{} & -0.5331294033713856 t & 0.1096316239241135 t & -0.7855025327869668 t & 1.209000312234239 t \\ [-0.5em]
			& +0.24812962236875004 & -1.000000000000000 & +1.688048335476923 & -0.06934455916442332
		\end{rktableau}
	$}
\end{equation*}

\subsubsection{SDIRK4(3)5}
\label{sec:SPC-MRI-GARK_methods:SDIRK4(3)5}
This method is based on SDIRK4M in \cite{Kennedy_2016_SDIRK-review}.
\begin{equation*}
	\begin{rktableau}{c|ccccc}
			\frac{1}{4} & \frac{1}{4} & {\scriptstyle 0} & {\scriptstyle 0} & {\scriptstyle 0} & {\scriptstyle 0} \\
			\frac{9}{10} & \frac{13}{20} & \frac{1}{4} & {\scriptstyle 0} & {\scriptstyle 0} & {\scriptstyle 0} \\
			\frac{2}{3} & \frac{580}{1287} & -\frac{175}{5148} & \frac{1}{4} & {\scriptstyle 0} & {\scriptstyle 0} \\
			\frac{3}{5} & \frac{12698}{37375} & -\frac{201}{2990} & \frac{891}{11500} & \frac{1}{4} & {\scriptstyle 0} \\
			{\scriptstyle 1} & \frac{944}{1365} & -\frac{400}{819} & \frac{99}{35} & -\frac{575}{252} & \frac{1}{4} \\
			\hline
			\text{} & \frac{487}{273}-\frac{142 t}{65} & -\frac{125 t}{182}-\frac{475}{3276} & \frac{297 t}{140}+\frac{99}{56} & -\frac{575}{252} & \frac{3 t}{4}-\frac{1}{8} \\
			\hline
			\text{} & \frac{357179 t}{270270} & \frac{222331 t}{72072} & \frac{1135934341 t}{442769040} & -\frac{11524110095 t}{1461137832} & \frac{636740663 t}{695779920} \\ [-0.2em]
			& +\frac{1}{27} & -\frac{17}{8} & +\frac{110483689}{63252720} & +\frac{28581755}{18975816} & -\frac{10434149}{63252720}
	\end{rktableau}
\end{equation*}

\subsubsection{ESDIRK4(3)6}
\label{sec:SPC-MRI-GARK_methods:ESDIRK4(3)6}
This method is based on ESDIRK4(3)6L[2]SA in \cite{Kennedy_2016_SDIRK-review} and has the property that the first and second column of coefficients are identical.
\begin{equation*}
	\resizebox{\linewidth}{!}{$
		\begin{rktableau}{c|cccccc}
			0 & 0 & 0 & 0 & 0 & 0 & 0 \\
			\frac{1}{2} & a\SS_{2,1} & 0.2500000000000000 & 0 & 0 & 0 & 0 \\
			\frac{2-\sqrt{2}}{4} & a\SS_{3,1} & -0.05177669529663688 & 0.2500000000000000 & 0 & 0 & 0 \\
			\frac{5}{8} & a\SS_{4,1} & -0.07655460838455727 & 0.5281092167691145 & 0.2500000000000000 & 0 & 0 \\
			\frac{26}{25} & a\SS_{5,1} & -0.7274063478261298 & 1.584995061740679 & 0.6598176339115803 & 0.2500000000000000 & 0 \\
			1 & a\SS_{6,1} & -0.01558763503571650 & 0.24876576709132033 & 0.5017726195721632 & -0.1082550204139335 & 0.2500000000000000 \\
			\hline
			\text{} & \gamma_1(t) & -6.163979155637189 t & 8.775315341826407 t & 2.197069503808978 t & 1.703312350342134 t & -0.24477388847031400 t \\ [-0.5em]
			& & +3.066401942782878 & -4.000000000000000 & -0.5967621323323260 & -0.9599111955850004 & +0.4238694423515700 \\
			\hline
			\text{} & \^\gamma_1(t) & -4.935764673620373 t & 7.151127236629060 t & 1.151758875793870 t & 3.303286684519598 t & -1.734643449701781 t \\ [-0.5em]
			& & +2.375000000000000 & -3.058823529411765 & -0.05607965938087753 & -1.734976675593132 & +1.099879864385774
		\end{rktableau}
	$}
\end{equation*}

\subsubsection{Stability plots}

\Cref{fig:SPC-MRI-GARK_scalar_stability,fig:SPC-MRI-GARK_matrix_stability} show the scalar and matrix stability regions, respectively.

\begin{figure}[ht!]
	\centering
	\begin{subfigure}{0.5\linewidth}
		\includegraphics[width=\linewidth]{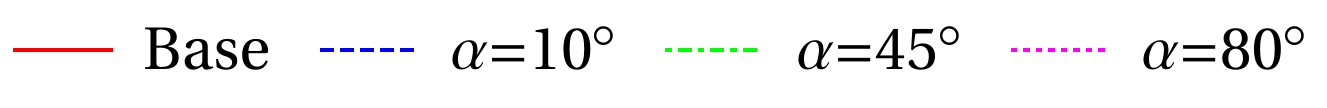}
	\end{subfigure}
	\\
	\scalarstabplot{SPC_SDIRK_2(1)2}{SDIRK2(1)2}
	\hfil
	\scalarstabplot{SPC_ESDIRK_2(1)3}{ESDIRK2(1)3}
	\hfil
	\scalarstabplot{SPC_SDIRK_3(2)4}{SDIRK3(2)4}
	\hfil
	\scalarstabplot{SPC_ESDIRK_3(2)4}{ESDIRK3(2)4}
	\hfil
	\scalarstabplot{SPC_SDIRK_4(3)5}{SDIRK4(3)5}
	\hfil
	\scalarstabplot{SPC_ESDIRK_4(3)6}{ESDIRK4(3)6}
	\caption{Scalar stability regions $\scalarstab[\infty]$ for \SPCabbv{} methods}
	\label{fig:SPC-MRI-GARK_scalar_stability}
\end{figure}

\begin{figure}[ht!]
	\centering
	\begin{subfigure}{0.65\linewidth}
		\includegraphics[width=\linewidth]{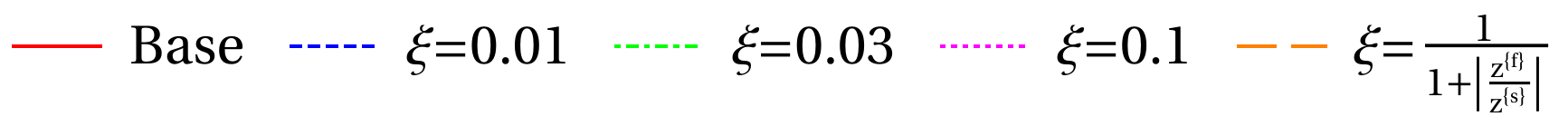}
	\end{subfigure}
	\\
	\matstabplot{SPC_SDIRK_2(1)2}{SDIRK2(1)2}
	\hfil
	\matstabplot{SPC_ESDIRK_2(1)3}{ESDIRK2(1)3}
	\hfil
	\matstabplot{SPC_SDIRK_3(2)4}{SDIRK3(2)4}
	\ifreport
		\end{figure}
		\begin{figure} \ContinuedFloat
	\fi
	\matstabplot{SPC_ESDIRK_3(2)4}{ESDIRK3(2)4}
	\hfil
	\matstabplot{SPC_SDIRK_4(3)5}{SDIRK4(3)5}
	\hfil
	\matstabplot{SPC_ESDIRK_4(3)6}{ESDIRK4(3)6}
	\caption{Matrix stability regions $\matstab[\infty][\ifreport \alpha \fi]$ for \SPCabbv{} methods}
	\label{fig:SPC-MRI-GARK_matrix_stability}
\end{figure}

\clearpage

\subsection{\IPCabbv{} methods}
\label{sec:IPC-MRI-GARK_methods}

We will using the following tableau representation for IPC-MRI-GARK methods:
\begin{equation*}
	\resizebox{\linewidth}{!}{$
		\begin{rktableau}{c|cccc|cccc}
			\cbase_1 & & & & & \psi_{1,1}(t) \\
			\cbase_2 & \gamma(t)_{2,1} & & & & \psi_{2,1}(t) & \psi_{2,2}(t) \\
			\vdots & \vdots & \ddots & & & \vdots & \vdots & \ddots \\
			\cbase_{s\S} & \gamma(t)_{s\S,1} & \ldots & \gamma(t)_{s\S,s\S-1} & & \psi_{s\S,1}(t) & \psi_{s\S,2}(t) & \ldots & \psi_{s\S,s\S}(t) \\ \hline
			& \^\gamma_1(t) & \ldots & \^\gamma_{s\S-1}(t) & 0 & \^\psi_1(t) & \^\psi_2(t) & \ldots & \^\psi_{s\S}(t)
		\end{rktableau}.
	$}
\end{equation*}

\subsubsection{SDIRK2(1)2}
\label{sec:IPC-MRI-GARK_methods:SDIRK2(1)2}
This method is based on the two stage, second order method in \cite{Alexander_1977_SDIRK}.
\begin{equation*}
	\begin{rktableau}{c|cc|cc}
		{\scriptstyle 1-}\frac{1}{\sqrt{2}} & {\scriptstyle 0} & {\scriptstyle 0} & {\scriptstyle 1-}\frac{1}{\sqrt{2}} & {\scriptstyle 0} \\
		{\scriptstyle 1} & \frac{1}{\sqrt{2}} & {\scriptstyle 0} & \frac{1}{\sqrt{2}}{\scriptstyle -1} & {\scriptstyle 1-}\frac{1}{\sqrt{2}} \\
		\hline
		\text{} & \frac{3}{5} & {\scriptstyle 0} & \frac{1}{\sqrt{2}}{\scriptstyle -1} & \frac{2}{5} \\
	\end{rktableau}
\end{equation*}

\subsubsection{ESDIRK2(1)3}
\label{sec:IPC-MRI-GARK_methods:ESDIRK2(1)3}
This method is based on TR-BDF2 in \cite{bank1985transient}.
\begin{equation*}
	\begin{rktableau}{c|ccc|ccc}
		{\scriptstyle 0} & {\scriptstyle 0} & {\scriptstyle 0} & {\scriptstyle 0} & {\scriptstyle 0} & {\scriptstyle 0} & {\scriptstyle 0} \\
		{\scriptstyle 2-\sqrt{2}} & {\scriptstyle 1-}\frac{1}{\sqrt{2}} & {\scriptstyle 0} & {\scriptstyle 0} & {\scriptstyle 0} & {\scriptstyle 1-}\frac{1}{\sqrt{2}} & {\scriptstyle 0} \\
		{\scriptstyle 1} & \frac{3}{2 \sqrt{2}}{\scriptstyle -1} & \frac{1}{2 \sqrt{2}} & {\scriptstyle 0} & {\scriptstyle 0} & \frac{1}{\sqrt{2}}{\scriptstyle -1} & {\scriptstyle 1-}\frac{1}{\sqrt{2}} \\
		\hline
		\text{} & \frac{1}{\sqrt{2}}-\frac{7}{10} & \frac{3}{10} & {\scriptstyle 0} & {\scriptstyle 0} & \frac{1}{\sqrt{2}}{\scriptstyle -1} & \frac{2}{5} \\
	\end{rktableau}
\end{equation*}

\subsubsection{SDIRK3(2)5}
\label{sec:IPC-MRI-GARK_methods:SDIRK3(2)5}

\begin{equation*}
	\begin{rktableau}{c|ccccc|ccccc}
		\frac{7}{40} & {\scriptstyle 0} & {\scriptstyle 0} & {\scriptstyle 0} & {\scriptstyle 0} & {\scriptstyle 0} & \frac{7}{40} & {\scriptstyle 0} & {\scriptstyle 0} & {\scriptstyle 0} & {\scriptstyle 0} \\
		\frac{1}{3} & \frac{19}{120} & {\scriptstyle 0} & {\scriptstyle 0} & {\scriptstyle 0} & {\scriptstyle 0} & -\frac{7}{40} & \frac{7}{40} & {\scriptstyle 0} & {\scriptstyle 0} & {\scriptstyle 0} \\
		\frac{1}{3} & \frac{1}{10} & -\frac{1}{10} & {\scriptstyle 0} & {\scriptstyle 0} & {\scriptstyle 0} & {\scriptstyle 0} & -\frac{7}{40} & \frac{7}{40} & {\scriptstyle 0} & {\scriptstyle 0} \\
		{\scriptstyle 1} & \frac{17341}{182400} & -\frac{73}{70} & \frac{687111}{425600} & {\scriptstyle 0} & {\scriptstyle 0} & {\scriptstyle 0} & {\scriptstyle 0} & -\frac{7}{40} & \frac{7}{40} & {\scriptstyle 0} \\
		{\scriptstyle 1} & -\frac{21487}{60800} & \frac{1618427}{1702400} & -\frac{1144471}{1702400} & \frac{3}{40} & {\scriptstyle 0} & {\scriptstyle 0} & {\scriptstyle 0} & {\scriptstyle 0} & -\frac{7}{40} & \frac{7}{40} \\
		\hline
		\text{} & \frac{2833}{60800} & -\frac{9}{35} & \frac{17257}{425600} & \frac{1}{6} & {\scriptstyle 0} & {\scriptstyle 0} & {\scriptstyle 0} & {\scriptstyle 0} & -\frac{7}{40} & \frac{107}{600} \\
	\end{rktableau}
\end{equation*}

\subsubsection{SDIRK4(3)6}
\label{sec:IPC-MRI-GARK_methods:ESDIRK4(3)6}

\begin{equation*}
	\resizebox{\linewidth}{!}{$
		\begin{rktableau}{c|cccccc|cccccc}
			\frac{1}{5} & {\scriptstyle 0} & {\scriptstyle 0} & {\scriptstyle 0} & {\scriptstyle 0} & {\scriptstyle 0} & {\scriptstyle 0} & \frac{1}{5} & {\scriptstyle 0} & {\scriptstyle 0} & {\scriptstyle 0} & {\scriptstyle 0} & {\scriptstyle 0} \\
			\frac{1}{4} & -\frac{73 t }{70} & {\scriptstyle 0} & {\scriptstyle 0} & {\scriptstyle 0} & {\scriptstyle 0} & {\scriptstyle 0} & \frac{73 t }{14} & -\frac{146 t }{35} & {\scriptstyle 0} & {\scriptstyle 0} & {\scriptstyle 0} & {\scriptstyle 0} \\ [-0.2em]
			& +\frac{4}{7} & & & & & & -\frac{393}{140} & +\frac{16}{7} & & & & \\
			\frac{1}{2} & -\frac{2592641 t }{425250} & \frac{32 t }{7} & {\scriptstyle 0} & {\scriptstyle 0} & {\scriptstyle 0} & {\scriptstyle 0} & \frac{454241 t }{85050} & -\frac{714082 t }{212625} & -\frac{16 t }{35} & {\scriptstyle 0} & {\scriptstyle 0} & {\scriptstyle 0} \\ [-0.2em]
			& +\frac{2253133}{425250} & -\frac{30}{7} & & & & & -\frac{454241}{170100} & +\frac{314516}{212625} & +\frac{3}{7} \\
			\frac{1}{2} & -\frac{79813 t }{26425} & \frac{417821 t }{79275} & -\frac{180296 t }{237825} & {\scriptstyle 0} & {\scriptstyle 0} & {\scriptstyle 0} & -\frac{23293 t }{1134} & \frac{20473 t }{945} & -\frac{2891 t }{9720} & -\frac{22537 t }{9720} & {\scriptstyle 0} & {\scriptstyle 0} \\ [-0.2em]
			& -\frac{5}{14} & -\frac{5}{6} & +\frac{4}{9} & & & & +\frac{23293}{2268} & -\frac{20473}{1890} & -\frac{997}{19440} & +\frac{5285}{3888} \\
			\frac{3}{4} & -\frac{1709523149 t }{68615910} & \frac{8462196 t }{449225} & \frac{2352991367 t }{1035014400} & \frac{180121 t }{143616} & {\scriptstyle 0} & {\scriptstyle 0} & \frac{7713555547 t }{310789710} & -\frac{1703745478 t }{70634025} & -\frac{3353446993 t }{1130144400} & \frac{137392139 t }{32289840} & \frac{360242 t }{639485} & {\scriptstyle 0} \\ [-0.2em]
			& +\frac{6626912}{467775} & -\frac{81}{7} & -\frac{8}{9} & -\frac{2}{11} & & & -\frac{7713555547}{621579420} & +\frac{851872739}{70634025} & +\frac{3353446993}{2260288800} & -\frac{30061615}{12915936} & -\frac{52224}{639485} \\
			{\scriptstyle 1} & \frac{1646963990099 t }{204132332250} & -\frac{78294288 t }{7636825} & \frac{49839881579 t }{17595244800} & -\frac{5075915 t }{2441472} & \frac{152 t }{165} & {\scriptstyle 0} & \frac{9215792648141 t }{449091130950} & -\frac{349735368626 t }{14580880875} & \frac{115392839939 t }{653223463200} & \frac{61269407807 t }{18663527520} & \frac{15316074 t }{10871245} & -\frac{76 t }{85} \\ [-0.2em]
			& -\frac{796870764337}{204132332250} & +\frac{113260367}{22910475} & -\frac{28671224497}{17595244800} & +\frac{4299217}{2441472} & -\frac{2}{3} & & -\frac{9215792648141}{898182261900} & +\frac{174867684313}{14580880875} & -\frac{9832286}{10871245} & -\frac{115392839939}{1306446926400} & -\frac{61269407807}{37327055040} & +\frac{11}{17} \\
			\hline
			\text{} & \frac{694507614551 t }{96062274000} & -\frac{9882343 t }{1078140} & \frac{13007509307 t }{4968069120} & -\frac{1639519 t }{940032} & \frac{49 t }{99} & {\scriptstyle 0} & \frac{13988077 t }{680400} & -\frac{10360601 t }{425250} & \frac{2 t }{15} & \frac{11 t }{3} & \frac{27 t }{20} & -\frac{7 t }{9} \\ [-0.2em]
			& -\frac{669461750351}{192124548000} & +\frac{9560707}{2156280} & -\frac{14640287027}{9936138240} & +\frac{2715895}{1880064} & -\frac{35}{99} & & -\frac{13988077}{1360800} & +\frac{10360601}{850500} & -\frac{1}{15} & -\frac{11}{6} & -\frac{7}{8} & +\frac{5}{9}
		\end{rktableau}
	$}
\end{equation*}

\subsubsection{Stability plots}
\Cref{fig:IPC-MRI-GARK_scalar_stability,fig:IPC-MRI-GARK_matrix_stability} show the scalar and matrix stability regions, respectively.

\begin{figure}[ht!]
	\centering
	\begin{subfigure}{0.5\linewidth}
		\includegraphics[width=\linewidth]{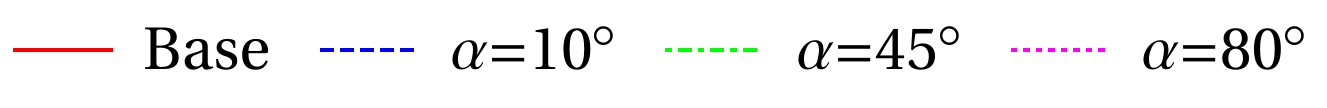}
	\end{subfigure}
	\\
	\scalarstabplot{IPC_SDIRK_2(1)2}{SDIRK2(1)2}
	\hfil
	\scalarstabplot{IPC_ESDIRK_2(1)3}{ESDIRK2(1)3}
	\hfil
	\scalarstabplot{IPC_SDIRK_3(2)5}{SDIRK3(2)5}
	\hfil
	\scalarstabplot{IPC_SDIRK_4(3)6}{SDIRK4(3)6}
	\caption{Scalar stability regions $\scalarstab[\infty]$ for \IPCabbv{} methods}
	\label{fig:IPC-MRI-GARK_scalar_stability}
\end{figure}

\begin{figure}[ht!]
	\centering
	\begin{subfigure}{0.65\linewidth}
		\includegraphics[width=\linewidth]{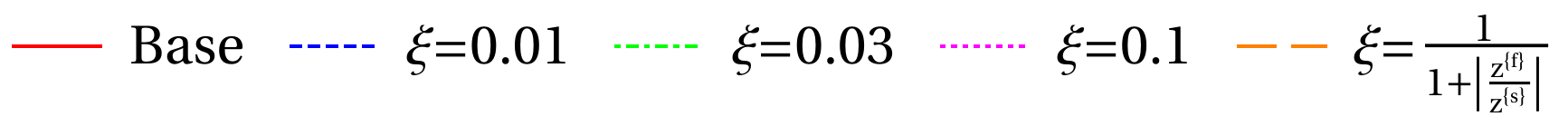}
	\end{subfigure}
	\\
	\matstabplot{IPC_SDIRK_2(1)2}{SDIRK2(1)2}
	\hfil
	\matstabplot{IPC_ESDIRK_2(1)3}{ESDIRK2(1)3}
	\hfil
	\matstabplot{IPC_SDIRK_3(2)5}{SDIRK3(2)5}
	\hfil
	\matstabplot{IPC_SDIRK_4(3)6}{SDIRK4(3)6}
	\caption{Matrix stability regions $\matstab[\infty][\ifreport \alpha \fi]$ for \IPCabbv{} methods}
	\label{fig:IPC-MRI-GARK_matrix_stability}
\end{figure}

\end{document}